\documentclass[10pt,leqno]{amsart}
   \usepackage[centertags]{amsmath}
   \usepackage{amsfonts}
   \usepackage{amsmath}
   \usepackage{amssymb}
   \usepackage{amsthm}
   \usepackage{newlfont}
   \usepackage[latin1]{inputenc}

\allowdisplaybreaks[3]

\theoremstyle{plain}
\newtheorem{theorem}{Theorem}[section]
\newtheorem{lemma}[theorem]{Lemma}

\newtheorem{corollary}[theorem]{Corollary}

\theoremstyle{definition}

\theoremstyle{remark}
\newtheorem{remark}[theorem]{Remark}
\newtheorem*{remark*}{Remark}

\numberwithin{equation}{section}

\newcommand\D{{\mathcal D}}
\newcommand\A{{\mathcal A}}

\newcommand\CC{{\mathbb C}}

\newcommand\RR{{\mathbb R}}
\newcommand\ZZ{{\mathbb Z}}

\newcommand\PP{{\mathbb P}}

\newcommand{\Cas}{\mathrm {Cas\,}}
\newcommand\p{\mbox{$\mathfrak{p}$}}
\newcommand\pp{\mbox{$\mathfrak{p}_{F}$}}

\newcommand\Sh{\mbox{\Large $\mathfrak {s}$}}

\title[]
{Constructing bispectral orthogonal polynomials  from the classical discrete families of Charlier, Meixner and Krawtchouk}

\author{Antonio J. Dur\'an and Manuel D. de la Iglesia}
\address{A. J. Dur\'an and Manuel D. de la Iglesia\\
Departamento de An\'{a}lisis Matem\'{a}tico \\
Universidad de Sevilla \\
Apdo (P. O. BOX) 1160\\
41080 Sevilla. Spain.}
\email{duran@us.es, mdi29@us.es}
\thanks{Partially supported by MTM2012-36732-C03-03 (Ministerio de Economía y Competitividad),
FQM-262, FQM-4643, FQM-7276 (Junta de Andalucía) and Feder Funds (European
Union).}



   \date{}
   
\begin{document}
   \maketitle

   \begin{abstract}
Given a sequence of polynomials $(p_n)_n$, an algebra of operators $\A$ acting in the linear space of polynomials and an operator $D_p\in \A$
with $D_p(p_n)=np_n$, we form a new sequence of polynomials $(q_n)_n$ by considering a linear combination of $m$ consecutive $p_n$: $q_n=p_n+\sum_{j=1}^m\beta_{n,j}p_{n-j}$. Using the concept of $\mathcal{D}$-operator, we determine the structure of the sequences $\beta_{n,j}, j=1,\ldots,m,$ in order that the polynomials $(q_n)_n$ are common eigenfunctions of an operator in the algebra $\A$. As an application, from the classical discrete families of Charlier, Meixner and Krawtchouk we construct orthogonal polynomials $(q_n)_n$
which are also eigenfunctions of higher order difference operators.
\end{abstract}

\section{Introduction}
The most important families of orthogonal polynomials are the classical, classical discrete or $q$-classical families (Askey scheme and its $q$-analogue). Besides the orthogonality, they are also common eigenfunctions of a second order differential, difference or $q$-difference operator, respectively.

The  issue of orthogonal polynomials which are also common eigenfunctions of a higher order differential operator was raised by H.L. Krall in 1939, when he obtained a complete classification for the case of a differential operator of order four (\cite{Kr2}). After his pioneer work, orthogonal polynomials which are also common eigenfunctions of higher order differential operators are usually called Krall polynomials. This terminology can be extended for finite order difference and $q$-difference operators. Hence Krall discrete polynomials are orthogonal polynomials which are also common eigenfunctions of higher order difference operators.

In the terminology introduced by Duistermaat and Gr\"unbaum (\cite{DG}; see also \cite{GrH1}, \cite{GrH3}), a family $(p_n)_n$ of Krall, Krall discrete or $q$-Krall polynomials is called bispectral. Indeed, as functions of the continuous variable $x$, the polynomials $(p_n)_n$ are eigenfunctions of a higher order differential, difference or $q$-difference operator, respectively. But, on the other hand, since the polynomials $(p_n)_n$ are orthogonal with respect to a measure, Favard's Theorem establishes that they satisfy a three term recurrence relation of the form ($p_{-1}=0$)
\begin{equation}\label{fvo}
xp_n(x)=a_{n+1}p_{n+1}(x)+b_np_n(x)+c_np_{n-1}(x), \quad n\ge 0,
\end{equation}
where $(a_n)_n$, $(b_n)_n$ and $(c_n)_n$ are sequences of real numbers with $a_{n}c_n\not =0$, $n\ge 1$ ($a_{n}c_n>0$, $n\ge 1$,
if we work with positive measures). The three term recurrence formula is saying that  the polynomials $p_n$, as functions of the discrete parameter $n$, are also eigenfunctions of a second order difference operator.

Since the eighties a lot of effort has been devoted to find Krall polynomials. Roughly speaking,
one can construct Krall polynomials $q_n(x)$, $n\ge 0$,  by using the Laguerre $x^\alpha e^{-x}$, or Jacobi weights $(1-x)^\alpha(1+x)^\beta$, assuming that one or two of the parameters $\alpha$ and $\beta$ are nonnegative integers and adding a linear combination of Dirac deltas and their derivatives at the endpoints of the orthogonality interval (\cite{Kr2}, \cite{koekoe}, \cite{koe}, \cite{koekoe2}, \cite{L1}, \cite{L2}, \cite{GrH1}, \cite{GrHH}, \cite{GrY}, \cite{Plamen1}, \cite{Plamen2}, \cite{Zh}). This procedure of adding deltas seems not to work if we want to construct Krall discrete polynomials from the classical discrete measures of Charlier, Meixner, Krawtchouk and Hahn (see the papers \cite{BH} and \cite{BK} by Bavinck, van Haeringen and Koekoek answering, in the negative,
a question posed by R. Askey in 1991).

We consider higher order difference operators of the form
\begin{equation}\label{doho}
D=\sum_{l=s}^rh_l\Sh _l, \quad s\le r, s,r\in \ZZ,
\end{equation}
where $h_l$ are polynomials and $\Sh_l$ stands for the shift operator $\Sh_l(p)=p(x+l)$. If $h_r,h_s\not =0$, the order of $D$ is then $r-s$. We also say that $D$ has genre $(s,r)$.

The first examples of orthogonal polynomials $(q_n)_n$ which are common eigenfunctions of a difference operator (\ref{doho}) of order bigger than $2$
have been introduced by one of us in \cite{du0}. Orthogonalizing measures for these families of polynomials are generated by multiplying the classical discrete weights by certain variants of the annihilator polynomial of a set of numbers. For a finite set of numbers $F$, this annihilator polynomial $\pp$ is defined by
\begin{equation*}\label{dp1i}
\pp(x)=\prod _{f\in F}(x-f).
\end{equation*}
The kind of transformation which consists in multiplying a measure $\mu$ by a polynomial $r$ is called a Christoffel transform.
It has a long tradition in the context of orthogonal polynomials: it goes back a century and a half ago when
E.B. Christoffel (see \cite{Chr} and also \cite{Sz}) studied it for the particular case $r(x)=x$.

Using this  annihilator polynomial (and other of its variants), a number of conjectures have been posed in \cite{du0}. The purpose of this paper is to prove those conjectures for the families of Charlier, Meixner and Krawtchouk polynomials. These families have in common that the eigenvalues for their associated second order difference operator are linear sequences in $n$.

The conjectures we will prove here are the following.
 Write $n_F$
for the number of elements of $F$.

\bigskip
\noindent
\textbf{Conjecture A.}
Let $\rho_{a} $ be the Charlier weight (see Section \ref{sch} for details). For a finite set $F$ of positive integers consider the measure
$$
\rho_a^F =\prod _{f\in F}(x-f)\rho_a ,
$$
and assume that the measure $\rho_a^F$ has an associated sequence $(q_n)_n$
of orthogonal polynomials (notice that, depending on the set $F$, the measure $\rho_a^F$ can be signed).
Then the polynomials $(q_n)_n$ are eigenfunctions of a difference operator of the form (\ref{doho}) with
$-s=r=\displaystyle \sum _{f\in F}f-\frac{n_F(n_F-1)}{2}+1.$

\bigskip
\noindent
\textbf{Conjecture B.}
Let $\rho_{a,c} $ be the Meixner weight (see Section \ref{sme} for details). For two finite sets $F_1, F_2$ of positive integers (the empty set is allowed, in which case $\pp(x)=1$)
consider  the measure
$$
\rho_{a,c}^{F_1,F_2}=\prod _{f\in F_1}(x+c+f)\prod _{f\in F_2}(x-f)\rho_{a,c}.
$$
Assume that the measure $\rho_{a,c}^{F_1,F_2}$ has an associated sequence $(q_n)_n$
of orthogonal polynomials. Then they are eigenfunctions of a difference operator of the form (\ref{doho}) with
$-s=r=\displaystyle r_{F_1}+r_{F_2}-1$ and $r_{F_i}=\displaystyle \sum _{f\in F_i}f-\frac{n_{F_i}(n_{F_i}-1)}{2}+1$.

\bigskip
\noindent
\textbf{Conjecture C.}
Let $\rho_{a,N} $ be the Krawtchouk weight (see Section \ref{skr} for details). For two finite sets $F_1, F_2$ of positive integers (the empty set is allowed) consider  the measure
$$
\rho_{a,N}^{F_1,F_2}=\prod _{f\in F_1}(x-f)\prod _{f\in F_2}(N-1-f-x)\rho_{a,N}.
$$
If $N$ is a positive integer, we assume that $f_{1,M},f_{2,M}<N/2$ (so that $F_1\cap \{ N-1-f,f\in F_2\}=\emptyset$),
where $f_{i,M}=\max F_i$.
Assume that the measure $\rho_{a,N}^{F_1,F_2}$ has an associated sequence $(q_n)_n$
of orthogonal polynomials. Then they are eigenfunctions of a difference operator of the form (\ref{doho}) with
$-s=r=\displaystyle r_{F_1}+r_{F_2}-1$ and $r_{F_i}=\displaystyle \sum _{f\in F_i}f-\frac{n_{F_i}(n_{F_i}-1)}{2}+1$.

\bigskip
We use three  ingredients to prove these conjectures.

\medskip
\noindent
\textbf{The first ingredient: $\D$-operators}. This is an abstract concept which has shown to be very useful to generate Krall, Krall discrete and $q$-Krall families of polynomials.
To define a $\D$-operator, we need a sequence of polynomials $(p_n)_n$, $\deg p_n=n$, and an algebra of operators $\A $ acting in the linear space of
polynomials $\mathbb{P}$.
In addition, we assume that the polynomials $p_n$, $n\ge 0$, are eigenfunctions of certain operator $D_p\in \A$ with eigenvalues that are linear in $n$: that is, we assume that $D_p(p_n)=np_n$, $n\ge 0$. Observe that no orthogonality conditions are imposed at this stage on the polynomials $(p_n)_n$.
Given  a sequence of numbers $(\varepsilon_n)_n$, a $\D$-operator $\D$ associated to the algebra $\A$ and the sequence of polynomials
$(p_n)_n$ is defined  by linearity in $\PP$ from $\D(p_n)=\sum _{j=1}^n (-1)^{j+1}\varepsilon_n\cdots \varepsilon_{n-j+1}p_{n-j}$, $n\ge 0$.
We then say that $\D$ is a $\D$-operator if $\D\in \A$.

Using $\D$-operators we can construct from the polynomials $(p_n)_n$ a huge class of families of polynomials $(q_n)_n$ which are also eigenfunctions of operators in the algebra $\A$. Indeed, assume we have $m$ $\D$-operators $\D_1, \D_2, \ldots, \D_m$ (not necessarily different) defined by the sequences $(\varepsilon _n^h)_n$, $h=1,\ldots , m$.

For the sake of simplicity, we assume here in the Introduction that for $h=1,2,\ldots,m$, the sequence $(\varepsilon_{n}^{h})_n$ is constant and non null in $n$ (actually, that is what happens for the $\D$-operators associated to the classical discrete families of Charlier, Meixner and Krawtchouk). We write $\xi_{x,i}^h$, $i\in\ZZ$ and $h=1,2,\ldots,m$, for the auxiliary functions defined by
\begin{equation*}\label{defxii}
\xi_{x,i}^h=\prod_{j=0}^{i-1}\varepsilon_{x-j}^{h}, \quad i\ge 1,\quad \quad \xi_{x,0}^h=1,\quad\quad \xi_{x,i}^h=\frac{1}{\xi_{x-i,-i}^h},\quad i\leq-1.
\end{equation*}

For $m$ arbitrary polynomials $R_1, R_2, \ldots, R_m$, we consider the sequence of polynomials $(q_n)_n$ defined by
\begin{equation}\label{qusi}
q_n(x)=\begin{vmatrix}
               p_n(x) & -p_{n-1}(x) & \cdots & (-1)^mp_{n-m}(x) \\
               \xi_{n,m}^1R_1(n) &  \xi_{n-1,m-1}^1R_1(n-1) & \cdots & R_1(n-m) \\
               \vdots & \vdots & \ddots & \vdots \\
                \xi_{n,m}^mR_m(n) &  \xi_{n-1,m-1}^mR_m(n-1) & \cdots & R_m(n-m)
             \end{vmatrix}.
\end{equation}
To ensure that the polynomial $q_n$, $n\ge 0$, has degree $n$ (this is necessary if we want the polynomials $(q_n)_n$ to be orthogonal), we assume that the (quasi) Casorati determinant
$$
\Omega (x)=\det \left(\xi_{x-j,m-j}^lR_l(x-j)\right)_{l,j=1}^m,
$$
satisfies that $\Omega(n)\not =0$, $n\ge 0$.

Notice that the dependence in $x$ in the determinant (\ref{qusi}) appears only in the first row, and hence $q_n$ is a linear combination of $m$ consecutive $p_n$'s. The magic of  $\D$-operators is that, whatever the polynomials $R_j$'s are, there always exists an operator $D_q$ in the algebra $\A$ for which the polynomials $q_n$, $n\ge 0$, are eigenfunctions. Moreover, the operator $D_q$ can be explicitly constructed from the operator $D_p$ using the $\D$-operators $\D_j$, $j=1,\ldots , m$. To stress the dependence of the polynomials $q_n$, $n\ge 0$, on the polynomials $R_j$, $j=1,\ldots , m$, we write $q_n=\Cas_{n}^{R_1,\ldots , R_m}$. Polynomials defined by  other similar forms of Casorati determinants have also a long tradition in the context of orthogonal polynomials and bispectral polynomials. Casorati determinants appear, for instance, to express orthogonal polynomials with respect to the Christoffel or Geronimus transform of a measure. See  \cite{Sz}, Th. 2.5 for the Christoffel transform and \cite{Gs}, \cite{Gs2} or \cite{Zh0} (and the references therein) for the Geronimus transform. The Geronimus transform associated to the polynomial $q(x)=(x-f_1)\cdots (x-f_k)$ is defined as follows: we say that $\tilde \mu$ is a Geronimus transform of $\mu$ if $q\tilde \mu =\mu$. Notice that the Geronimus transform is reciprocal of the Christoffel transform. Geronimus transform are sometimes called Darboux transform. The reason is the following. As we have already mentioned, the three term recurrence relation (\ref{fvo}) for the orthogonal polynomials with respect to $\mu$ can be rewritten as $xp_n=J(p_n)$, where $J$ is the second order difference operator $J=a_{n+1}\Sh_1+b_n\Sh_0+c_n\Sh_{-1}$ (acting on the discrete variable $n$).
For any $\lambda \in \CC $, decompose $J$ into $J=AB+\lambda I$
whenever it is possible, where $A=\alpha_n\Sh_0+\beta_n\Sh_1$ and $B= \delta_n\Sh_{-1}+\gamma_n\Sh_0$. We then call $\tilde J=BA+\lambda I$
a Darboux transform of $J$ with parameter $\lambda$. It turns out that the second order difference operator $\tilde J$ associated to a Geronimus transform $\tilde \mu$ of $\mu$ can be obtained by applying a sequence of  $k$ Darboux transforms (with parameters $f_i$, $i=1,\ldots, k$) to the operator $J$ associated to the measure $\mu$. This kind of Darboux transform has been used by Gr\"unbaum, Haine, Hozorov, Yakimov or Iliev to construct Krall and $q$-Krall polynomials. For the particular cases of Laguerre, Jacobi or Askey-Wilson polynomials, one can found Casorati determinants similar to (\ref{qusi}) in \cite{GrHH}, \cite{GrY}, \cite{HP}, \cite{Plamen1} or \cite{Plamen2}.

The concept of $\D$-operators have been introduced in \cite{du1} by one of us (see also \cite{AD}). Among other things, it has been used to prove the particular cases of conjectures A, B and C above when the finite sets of positive integers involved are of the form  $\{1, 2, 3, \ldots, k\}$. This
particular case corresponds with a determinant of size $2\times 2$ ($m=1$) in (\ref{qusi}) (in other words, the polynomials $q_n$ (\ref{qusi}) are a linear combination of two consecutive $p_n$'s). Information about $\D$-operators for the classical discrete families of Charlier, Meixner and Krawtchouk is summarized in the following table.

\bigskip
\begin{center}
\begin{tabular}{||l |c|c| r||}
\hline
Classical discrete family & \# $\D$-operators & $\varepsilon_n$  & $\D$-operators \\
\hline
Charlier & 1 & $1$ & $\D=\nabla$ \\
\hline
Meixner & 2 & $\frac{a}{1-a}$ & $\D_1=\frac{a}{1-a}\Delta$ \\
\hline
 &  & $\frac{1}{1-a}$ & $\D_2=\frac{1}{1-a}\nabla$ \\
\hline
Krawtchouk & 2 & $\frac{1}{1+a}$ & $\D_1=\frac{1}{1+a}\nabla$ \\
\hline
 &  & $\frac{-a}{1+a}$ & $\D_2=\frac{-a}{1+a}\Delta$ \\
\hline
\end{tabular}
\end{center}
\bigskip

\noindent
\textbf{The second ingredient.} With the second ingredient, orthogonality with respect to a measure enters into the picture. Indeed, even if we assume that the polynomials $(p_n)_n$ are orthogonal, only for a convenient choice of the polynomials $R_j$, $j=1,\ldots, m$, the polynomials (\ref{qusi}) $q_n=\Cas_n^{R_{1},\ldots , R_{m}}$, $n\ge 0$, are also orthogonal with respect to a measure. We now assume that the polynomials $(p_n)_n$ are any of the classical discrete families of Charlier, Meixner or Krawtchouk. With this assumption the second ingredient establishes how to chose the polynomials $R_j$'s such that the polynomials $q_n=\Cas_n^{R_1,\ldots ,R_m}$ (\ref{qusi}) are also orthogonal with respect to a measure. This second ingredient turns into a very nice symmetry between the family $(p_n)_n$ and the polynomials $R_j$'s. Indeed, for $m=1$, it has been shown in \cite{du1} that the polynomial $R_1$ can be chosen in the same family as $(p_n)_n$ but with different parameters. We will see here that  the same choice also works for an arbitrary positive integer $m$. More precisely, given a $\D$-operator for the family $(p_n)_n$ and a nonnegative integer $j$ we provide a polynomial $R_j$ of degree $j$ such that
for any different nonnegative integers $g_1,\ldots , g_m$, the polynomials $\Cas_n^{R_{g_1},\ldots , R_{g_m}}$, $n\ge 0$, are orthogonal with respect to a measure. The choice of these polynomials $R_j$ is summarized in the following table, where the symmetry mentioned above is explicit.

\bigskip
\begin{center}
\begin{tabular}{||l |c| r||}
\hline
Classical discrete family & $\D$-operators & $R_j(x)$ \\
\hline
Charlier: $c_n^a$, $n\ge 0$ & $\nabla$ & $c_j^{-a}(-x-1)$, $j\ge 0$  \\
\hline
Meixner: $m_n^{a,c}$, $n\ge 0$ & $\frac{a}{1-a}\Delta$ & $m_j^{1/a,2-c}(-x-1)$, $j\ge 0$  \\
\hline
 & $\frac{1}{1-a}\nabla$ & $m_j^{a,2-c}(-x-1)$, $j\ge 0$ \\
\hline
Krawtchouk: $k_n^{a,N}$, $n\ge 0$ &  $\frac{1}{1+a}\nabla$ & $k_j^{a,-N}(-x-1)$, $j\ge 0$  \\
\hline
 &  $\frac{-a}{1+a}\Delta$ & $k_j^{1/a,-N}(-x-1)$, $j\ge 0$ \\
\hline
\end{tabular}
\end{center}
\bigskip

\noindent
\textbf{The third ingredient.} We still need a last ingredient for identifying the measure $\tilde \rho$ with respect to which the polynomials $q_n=\Cas_n^{R_{g_1},\ldots ,R_{g_m}}$ (\ref{qusi}) are orthogonal. These polynomials depend on the set of indices $G=\{g_1,\ldots , g_m\}$ (the degrees of the polynomials $R_{g_j}$). It turns out that the orthogonalizing measure for the polynomials $(\Cas_n^{R_{g_1},\ldots ,R_{g_m}})_n$ is one of the measures which appear in conjectures A, B and C above. These measures depend on certain finite sets $F$'s of positive integers. The third ingredient establishes the relationship between these sets $F$'s and the set $G$. This relationship is given by suitable transforms defined in the set $\Upsilon$ formed by all finite sets of positive integers. One of these transforms is the involution defined by
\begin{equation}\label{dinvi}
I(F)=\{1,2,\ldots, f_k\}\setminus \{f_k-f,f\in F\},
\end{equation}
where $f_k=\max F$ and $k$ the number of elements of $F$.

\medskip

The contents of the paper are as follows. The first ingredient (that is, how to use $\D$-operators for constructing the polynomials $(q_n)_n$ (\ref{qusi}) and its associated operator $D_q\in \A$ for which they are eigenfunctions) is explained in Section 3. The proofs for this
section, which can be quite technical at certain points, are given separately in Section \ref{sproofs}. An explicit construction of the operator $D_q$ is included. This allows us to find its order when we particularize for the classical discrete families of Charlier, Meixner and Krawtchouk.
Section \ref{ssi} is devoted to the second and third ingredients. In particular, we develop a strategy to find the polynomials $R_j$, $j\ge 0$, for which the polynomials $q_n=\Cas_n^{R_{g_1},\ldots ,R_{g_m}}$ (\ref{qusi}) are also orthogonal with respect to a measure. This is completed in Sections \ref{sch}, \ref{sme} and \ref{skr}, where we study the particular cases of Charlier, Meixner and Krawtchouk, respectively. In each case, we prove the corresponding conjectures A, B and C above.

To make this introduction more useful to the reader, we include here the main result for Charlier polynomials (conjecture A above is an easy corollary of it).

\begin{theorem}\label{mthch} Let $F$ be a finite set of positive integers and consider the involuted set $I(F)=G=\{ g_1,\ldots, g_m\}$ with $g_i<g_{i+1}$, where the involution $I$ is defined by (\ref{dinvi}). Let $\rho_a$ be the Charlier measure (\ref{Chw}) and $(c_n^a)_n$ its sequence of orthogonal polynomials defined by (\ref{Chpol}). Assume that $\Omega_G (n)\not =0$, $n\ge 0$, where the $m\times m$ Casorati determinant $\Omega _G$ is defined by
$$
\Omega_G (n)=\det \left(c^{-a}_{g_l}(-n-j-1)\right)_{l,j=1}^m.
$$
We then define the sequence of polynomials $(q_n)_n$ by
\begin{equation}\label{quschi}
q_n(x)=\begin{vmatrix}
c^a_n(x) & -c^a_{n-1}(x) & \cdots & (-1)^mc^a_{n-m}(x) \\
c^{-a}_{g_1}(-n-1) & c^{-a}_{g_1}(-n) & \cdots &
c^{-a}_{g_1}(-n+m-1) \\
               \vdots & \vdots & \ddots & \vdots \\
               c^{-a}_{g_m}(-n-1) & \displaystyle
               c^{-a}_{g_m}(-n) & \cdots &c^{-a}_{g_m}(-n+m-1)
             \end{vmatrix}.
\end{equation}
Then

\noindent
(1) The polynomials $(q_n)_n$ are orthogonal with respect to the measure
$$
\tilde \rho_a^F=\prod_{f\in F}(x+f_k+1-f)\rho_a(x+f_k+1),
$$
where $f_k=\max F$ and $k$ the number of elements of $F$.

\noindent
(2) The polynomials $(q_n)_n$ are eigenfunctions of a higher order difference operator of the form (\ref{doho}) with
$-s=r=\displaystyle \sum _{f\in F}f-\frac{k(k-1)}{2}+1$
(which we explicitly construct).
\end{theorem}

We stress the importance that symmetries have in this paper. They appear explicitly in the second ingredient above (see, for instance, the formula (\ref{quschi})) but also they are implicit in the first ingredient (for instance, Schur symmetric functions play an important role in the computation of the order of the operators with respect to which the polynomials $(q_n)_n$ (\ref{qusi}) are eigenfunctions).

\section{Preliminaries}
For a linear operator $D:\PP \to \PP$ and a polynomial $P(x)=\sum _{j=0}^ka_jx^j$, the operator $P(D)$ is defined in the usual way
$P(D)=\sum _{j=0}^ka_jD^j$.

Let $\mu$ be a moment functional on the real line, that is, a linear mapping $\mu :\PP \to \RR$.
The $n$-th moment of $\mu $ is defined by $\mu_n=\langle \mu, x^n\rangle $.
It is well-known that any moment functional on the real line can be represented by integrating with respect to a Borel measure
(positive or not) on the real line
(this representation is not unique \cite{du-1}).
If we also denote this measure by $\mu$, we have $\langle \mu,p\rangle=\int p(x)d\mu(x)$ for all polynomial $p\in \PP$. Taking this into account,
we will conveniently use along this paper one or other terminology (orthogonality with respect to a moment functional or with
respect to a measure). We say that a sequence of polynomials $(p_n)_n$, $p_n$ of degree $n$, $n\ge 0$, is orthogonal with respect to the moment functional $\mu$ if $\langle \mu, p_np_m\rangle=0$, for $n\not =m$ and $\langle \mu, p_n^2\rangle\not =0$.

\bigskip
The kind of transformation which consists in multiplying a moment functional $\mu$ by a polynomial $r$ is called a Christoffel transform. The new
moment functional $r\mu$ is defined by $\langle r\mu,p\rangle =\langle \mu,rp\rangle $.

For a real number $\lambda$, the moment functional $\mu (x+\lambda)$ is defined in the usual way
$\langle \mu (x+\lambda),p\rangle =\langle \mu ,p(x-\lambda)\rangle $. Hence, if $(p_n)_n$ are orthogonal polynomials with respect
to $\mu$ then $(p_n(x+\lambda))_n$ are orthogonal with respect to $\mu(x+\lambda)$.

We will use the following straightforward Lemma in relation with the second ingredient of our strategy (see the Introduction).

\begin{lemma} Let $F$ be a finite set of different complex numbers $F=\{f_i,i=1,\ldots , k\}$. Write $\p _F(x)=\prod_{i=1}^k(x-f_i)$. Then
for a monic polynomial $s$
\begin{equation}\label{sumuif}
\sum_{i=1}^k\frac{s(f_i)}{\p '_F(f_i)}=\begin{cases} 0,&\mbox{for $\deg (s) \le k-2$,}\\
1,&\mbox{for $\deg(s)=k-1$.}\end{cases}
\end{equation}
\end{lemma}

\begin{proof}
It is enough to take into account that $s(f_i)/\p '_F(f_i)$, $i=1,\ldots ,k $, are the residues of the rational function $s(z)/\p_F(z)$.
\end{proof}

\section{$\D$-operators}

The concept of $\D$-operator was introduced by one of us in the paper \cite{du1}. In  \cite{du1} and \cite{AD} it has been showed that $\D$-operators turn out to be an extremely useful tool of a unified method
to generate families of polynomials which are eigenfunctions of higher order differential, difference or $q$-difference operators.
Hence, we start by reminding the concept of $\D$-operator.

The starting point is a sequence of polynomials $(p_n)_n$, $\deg p_n=n$, and an algebra of operators $\A $ acting in the linear space of
polynomials $\mathbb{P}$. For the classical discrete polynomials, we will consider the algebra $\A$ formed by all finite order difference operators of the form (\ref{doho})
\begin{equation}\label{algdiffd}
\A =\left\{ \sum_{j=s}^rh_j\mathfrak{s}_j : h_j\in \PP, j=s,\ldots,r, s\leq r \right\}.
\end{equation}

In addition, we assume that the polynomials $p_n$, $n\ge 0$, are eigenfunctions of certain operator $D_p\in \A$. We write $(\theta_n)_n$ for the corresponding eigenvalues, so that $D_p(p_n)=\theta_np_n$, $n\ge 0$. In this paper we only consider the case when the sequence of eigenvalues $(\theta _n)_n$ is linear in $n$. For $\D$-operators associated to polynomials $(p_n)_n$ for which the sequence of eigenvalues $(\theta _n)_n$ is not linear in $n$ see \cite{du1}.

Given  a sequence of numbers $(\varepsilon_n)_n$, a $\D$-operator associated to the algebra $\A$ and the sequence of polynomials
$(p_n)_n$ is defined as follows.
We first consider  the operator $\D :\PP \to \PP $ defined by linearity
from
\begin{equation}\label{defTo}
\D (p_n)=\sum _{j=1}^n (-1)^{j+1}\varepsilon _n\cdots \varepsilon _{n-j}p_{n-j},\quad n\ge 0.
\end{equation}
We then say that $\D$ is a $\D$-operator if $\D\in \A$.

The following Lemma was proved in \cite{du1} and shows how to use $\D$-operators to construct new sequences of polynomials $(q_n)_n$ such that there exists an operator $D_q\in \A$ for which they are eigenfunctions.

\begin{lemma}(Lemma 3.2 of \cite{du1})\label{fl1v}
Let $\A$ and $(p_n)_n$ be, respectively, an algebra of operators acting in the linear space of polynomials, and a sequence of polynomials $(p_n)_n$,
$\deg p_n=n$. We assume that there exists an operator $D_p\in \A$ satisfying that
$D_p(p_n)=np_n$, $n\ge 0$. We also have a sequence of numbers $(\varepsilon_n )_n$ which defines a $\D$-operator  for $(p_n)_n$
and $\A$ (see \eqref{defTo}).
For an arbitrary polynomial $R$ such that $R(n)\not=0$, $n\ge 0$, we define a new polynomial $P$ by $P(x)-P(x-1)=R(x)$.
We finally define the sequence of polynomials $(q_n)_n$ by $q_0=1$ and
\begin{equation*}\label{defqng}
q_n=p_n+\beta_np_{n-1},\quad n\ge 1,
\end{equation*}
where the numbers $\beta_n$, $n\ge 0$, are given by
\begin{equation*}\label{defbetng}
\beta_n=\varepsilon_n \frac{R(n)}{R(n-1)}, \quad n\ge 1.
\end{equation*}
Then $D_q(q_n)=P(n)q_n$ where the operator $D_q$ is defined by
\begin{equation*}\label{defD}
D_q=P(D_p)+\D R(D_p).
\end{equation*}
Moreover $D_q\in \A$.
\end{lemma}

The purpose of this Section is to extend the method in the previous Lemma for a combination of $m$, $m\ge 1$, consecutive $p_n$'s.

Instead of the polynomial $R$ and the $\D$-operator $\D$ of Lemma \ref{fl1v}, we now use $m$ arbitrary polynomials $R_1, R_2, \ldots, R_m$ and $m$ $\D$-operators $\D_1, \D_2, \ldots, \D_m$ (not necessarily different) defined by the sequences $(\varepsilon _n^h)_n$, $h=1,\ldots , m$:
\begin{equation}\label{Dh}
\mathcal{D}_h(p_n)=\sum _{j=1}^n (-1)^{j+1}\varepsilon_{n}^{h}\cdots\varepsilon_{n-j+1}^{h}p_{n-j},\quad h=1,2,\ldots,m.
\end{equation}

We will assume that for $h=1,2,\ldots,m$, the sequence $(\varepsilon_{n}^{h})_n$ is a rational function in $n$. Actually, in the examples we will study in this paper, these sequences are constant in $n$. We write $\xi_{x,i}^h$, $i\in\ZZ$ and $h=1,2,\ldots,m$, for the auxiliary functions defined by
\begin{equation}\label{defxi}
\xi_{x,i}^h=\prod_{j=0}^{i-1}\varepsilon_{x-j}^{h}, \quad i\ge 1,\quad \quad \xi_{x,0}^h=1,\quad\quad \xi_{x,i}^h=\frac{1}{\xi_{x-i,-i}^h},\quad i\leq-1.
\end{equation}
We will consider the $m\times m$ (quasi) Casorati determinant defined by
\begin{equation}\label{casd1}
\Omega (x)=\det \left(\xi_{x-j,m-j}^lR_l(x-j)\right)_{l,j=1}^m.
\end{equation}

The details of our method are included in the following Theorem (which it will be proved in Section \ref{sproofs}).

\begin{theorem}\label{Teor1} Let $\A$ and $(p_n)_n$ be, respectively, an algebra of operators acting in the linear space of polynomials, and a sequence of polynomials $(p_n)_n$, $\deg p_n=n$. We assume that $(p_n)_n$ are eigenfunctions of an operator $D_p\in \A$ with eigenvalues equal to $n$, that is, $D_p(p_n)=np_n$, $n\ge 0$. We also have $m$ sequences of numbers $(\varepsilon_{n}^{1} )_n,\ldots,(\varepsilon_{n}^{m} )_n$, which define $m$ $\D$-operators $\D_1,\ldots,\D_m$ (not necessarily different) for $(p_n)_n$ and $\A$ (see \eqref{Dh})) and assume that for $h=1,2,\ldots,m$, each sequence $(\varepsilon_{n}^{h})_n$ is a rational function in $n$ which does not vanish for $n\in \ZZ$.

Let $R_1, R_2, \ldots, R_m$ be $m$ arbitrary polynomials satisfying that $\Omega (n)\not =0$, $n\ge 0$, where $\Omega $ is the Casorati determinant defined by \eqref{casd1}.

Consider the sequence of polynomials $(q_n)_n$ defined by
\begin{equation}\label{qus}
q_n(x)=\begin{vmatrix}
               p_n(x) & -p_{n-1}(x) & \cdots & (-1)^mp_{n-m}(x) \\
               \xi_{n,m}^1R_1(n) &  \xi_{n-1,m-1}^1R_1(n-1) & \cdots & R_1(n-m) \\
               \vdots & \vdots & \ddots & \vdots \\
                \xi_{n,m}^mR_m(n) &  \xi_{n-1,m-1}^mR_m(n-1) & \cdots & R_m(n-m)
             \end{vmatrix}.
\end{equation}
For a rational function $S$ and $h=1,\ldots,m$, we define the function $M_h(x)$ by
\begin{align}\label{emeiexp}
\nonumber M_h(x)&=\sum_{j=1}^m(-1)^{h+j}\xi_{x,m-j}^hS(x+j)\times\\
&\qquad\times\det\left(\xi_{x-r,m-j-r}^{l}R_l(x-r)\right)_{\scriptsize \left\{\begin{array}{l}
                                                           l\in\{1,2,\ldots,m\}\setminus\{h\} \\
                                                            r\in\{-j+1,-j+2,\ldots,m-j\}\setminus\{0\}
                                                          \end{array}\right\}}.
\end{align}
If we assume that the functions $S(x)\Omega (x)$ and $M_h(x)$, $h=1,\ldots,m$, are polynomials in $x$
then there exists an operator $D_{q,S}\in \A$ such that
$$
D_{q,S}(q_n)=\lambda_nq_n,\quad n\ge 0.
$$
Moreover, an explicit expression of this operator can be displayed. Indeed, write  $P_S$ for the polynomial defined by
\begin{equation}\label{Pgs}
P_S(x)-P_S(x-1)=S(x)\Omega (x).
\end{equation}
Then the operator $D_{q,S}$ is defined by
\begin{equation}\label{Dq}
D_{q,S}=P_S(D_p)+\sum_{h=1}^mM_h(D_p)\D_hR_h(D_p),
\end{equation}
where $D_p\in \A$ is the operator for which the polynomials $(p_n)_n$ are eigenfunctions. Moreover
$\lambda_n=P_S(n)$.

\end{theorem}

Notice that the dependence in $x$ of the polynomials (\ref{qus}) appears only in the first row, and hence $q_n$ is a linear combination of $m$ consecutive $p_n$'s.

\begin{remark}
In Sections \ref{sch}, \ref{sme} and \ref{skr}, we will apply Theorem \ref{Teor1} to the Charlier, Meixner and Krawtchouk polynomials. We will see there that
the polynomial $P_S$ (see \eqref{Pgs}) will give the order of the difference operator $D_{q,S}$ \eqref{Dq} with respect to which the new polynomials $(q_n)_n$ are eigenfunctions. This will be a consequence of the following Lemmas (which will be also proved in Section \ref{sproofs}).
\end{remark}

\begin{lemma}\label{lgp1} For two nonnegative integers $m_1$ and $m_2$ (at least one different to zero) write $m=m_1+m_2$ and let $R_1, R_2, \ldots, R_m,$ be non null polynomials satisfying that $\deg R_i\not =\deg R_{j}$, for $i\not =j$ and $1\le i,j \le m_1$ or  $m_1+1\le i,j \le m$. Let $l_1, \ldots , l_m,$ be numbers satisfying that $l_i\not =l_j$, $i\not =j$. For a complex number $w$, define new polynomials $S_{i,j}(x)$, $i,j=1,\ldots , m$, by
$$
S_{i,j}(x)=\begin{cases} R_i(x-l_j),& \mbox{if $1\le i \le m_1$,}\\
w ^{m-j}R_i(x-l_j),& \mbox{if $m_1+1\le i \le m$.}
\end{cases}
$$
Except for a finite set of complex numbers $w$ (which depends on $l_i$, $i=1,\ldots, m$), we have
$$
\deg \left(\det (S_{i,j})_{i,j=1}^m\right) =\left(\sum_{i=1}^m\deg R_i\right) -\binom{m_1}{2}-\binom{m_2}{2}.
$$
Moreover, if $l_i=i$, this finite set of complex numbers $w$ is  $\{0, 1\}$.
\end{lemma}

\bigskip

\begin{lemma}\label{lgp2} For two nonnegative integers $m_1$ and $m_2$ (at least one different to zero) write $m=m_1+m_2$ and let $R_1, R_2, \ldots, R_m,$ be non null polynomials satisfying that $\deg R_i\not =\deg R_{j}$, for $i\not =j$ and $1\le i,j \le m_1$ or  $m_1+1\le i,j \le m$. For a complex number $w$, write $U_j(x)$, $j=1,\ldots , m+1$, for the matrix with row $(U_j(x))_l$ equal to
$$
(U_j(x))_l=\begin{cases} (R_l(x-r))_{{\scriptsize
r=1-j,2-j,\ldots,m+1-j, r\not =0}},& 1\le l\le m_1  \\
(w^{m+1-j-r}R_{l}(x-r))_{{\scriptsize
r=1-j,2-j,\ldots,m+1-j, r\not =0}},& m_1+1\le l\le m.\end{cases}
$$
Then for $\epsilon=0,1$ we have
\begin{align*}
&\deg \left(\sum _{j=1}^{m+1} (-1)^{j+1}w^{\epsilon (m+1-j)}\det U_j(x)\right) \le \left(\sum_{i=1}^m\deg R_i\right)
\\&\hspace{4cm}-\max \left\{\binom{m_1}{2}+\binom{m_2+1}{2},\binom{m_1+1}{2}+\binom{m_2}{2} \right\}.
\end{align*}
\end{lemma}

\bigskip

\section{When are the polynomials $(\Cas_n^{R_1,\ldots ,R_m})_n$ orthogonal?} \label{ssi}
Only for a convenient choice of the polynomials $R_j$, $j=1,\ldots, m$, the polynomials $(q_n)_n$ (\ref{qus}) are also orthogonal with respect to a measure. As we wrote in the Introduction, when the sequence $(p_n)_n$ is any of the classical discrete families of Charlier, Meixner and Krawtchouk, a very nice symmetry between the family $(p_n)_n$ and the polynomials $R_j$'s appears. Indeed, for $m=1$, it has been shown in \cite{du1} that the polynomial $R_1$ can be chosen in the same family as $(p_n)_n$ but with different parameters (see the second table in the Introduction). The purpose of this section is to provide the tools to prove that the same choice also works for an arbitrary positive integer $m$ (this was called second ingredient in the Introduction).

The key is given by the recurrence formula
\begin{equation}\label{relaR0}
\varepsilon_{n+1}a_{n+1}R_j(n+1)-b_nR_j(n)+\frac{c_n}{\varepsilon_n}R_j(n-1)=(\eta j+\kappa )R_j(n),\quad n\in \ZZ,
\end{equation}
where $\eta $ and $\kappa$ are numbers independent of $n$ and $j$, $a_n, b_n, c_n$ are the recurrence sequences for the orthogonal polynomials $(p_n)_n$ and $(\varepsilon _n)_n$ is a
sequence which defines a $\D$-operator for $(p_n)_n$ and the algebra $\A$.

Indeed, given the sequences (in the integers) $(a_n)_{n\in \ZZ }$, $(b_n)_{n\in \ZZ }$, $(c_n)_{n\in \ZZ }$ and
$(\varepsilon_{n}^{h})_{n\in \ZZ }$, $h=1,\ldots ,m$, with $\varepsilon_{n}^{h}\not =0$, $n\in\ZZ$, and numbers $\eta _h, \kappa _h$, $h=1,\ldots ,m$, assume we have for each $j\ge 0$ and $h=1,\ldots ,m$, one more sequence denoted by $(R_{j}^{h}(n))_{n\in \ZZ }$ satisfying
\begin{equation}\label{relaR}
\varepsilon_{n+1}^{h}a_{n+1}R_{j}^{h}(n+1)-b_nR_{j}^{h}(n)+\frac{c_n}{\varepsilon_{n}^{h}}R_{j}^{h}(n-1)=(\eta _h j+\kappa _h)R_{j}^{h}(n),\quad n\in \ZZ.
\end{equation}
Assume also that the polynomials $(p_n)_n$ are orthogonal with respect to a measure and satisfy the three term recurrence relation ($p_{-1}=0$)
\begin{equation}\label{ttrr}
xp_n=a_{n+1}p_{n+1}(x)+b_np_n(x)+c_np_{n-1}(x), \quad n\ge 0.
\end{equation}
Define the auxiliary numbers $\xi _{n,i}^h$, $i\ge 0$, $n\in \ZZ $ and $h=1,\ldots ,m$, by
\begin{equation*}\label{defxi2}
\xi_{n,i}^h=\prod_{j=n-i+1}^{n}\varepsilon_{j}^{h},\quad i\ge 1,\quad \quad \xi_{n,0}^h=1,\quad\quad \xi_{x,i}^h=\frac{1}{\xi_{x-i,-i}^h},\quad i\leq-1.
\end{equation*}
Notice that for $x=n$, the number $\xi_{n,i}^h$ coincides with the number defined by (\ref{defxi}).


Given a $m$-tuple $G$ of $m$ positive integers, $G=(g_1,\ldots , g_m)$, assume that the $m$ numbers
\begin{equation}\label{diffn}
\tilde g_h=\eta _hg_h+\kappa_h,\quad h=1,\ldots ,m,
\end{equation}
are different. Call $\tilde G=\{ \tilde g_1,\ldots , \tilde g_m\}$. Consider finally the $m\times m$ Casorati determinant $\Omega _G$ defined by
\begin{equation*}\label{casort}
\Omega_G (n)=\det \left(\xi_{n,m-j}^lR_{g_l}^l(n-j)\right)_{l,j=1}^m,
\end{equation*}
and assume that $\Omega_G (n)\not =0$ for $n\ge 0$.

We then define the sequence of polynomials $(q_n^G)_n$ by
\begin{equation}\label{quso}
q_n^G(x)=\begin{vmatrix}
p_n(x) & -p_{n-1}(x) & \cdots & (-1)^mp_{n-m}(x) \\
\displaystyle\xi_{n,m}^1R_{g_1}^1(n) & \xi_{n-1,m-1}^1R_{g_1}^1(n-1) & \cdots &
R_{g_1}^1(n-m) \\
               \vdots & \vdots & \ddots & \vdots \\
              \xi_{n,m}^mR_{g_m}^m(n) & \displaystyle
               \xi_{n-1,m-1}^mR_{g_m}^m(n-1) & \cdots &R_{g_m}^m(n-m)
             \end{vmatrix}.
\end{equation}
Notice that if for each $h=1,\ldots , m$,  $R_{g_h}^h(n)$, $n\in \ZZ$, is a polynomial in $n$, then the polynomials $q_n^G$ (\ref{quso})
fit into the definition of the polynomials (\ref{qus}) in Theorem \ref{Teor1}, and hence they are eigenfunctions of a certain operator in the algebra $\A$.

The key to prove that the polynomials $(q_n^G)_n$ are orthogonal with respect to a measure $\tilde \rho $ are the following formulas: there exists a constant $c_G\not =0$ such that
\begin{align}\label{foeq1}
\langle \tilde \rho,p_n\rangle &=(-1)^n c_G\sum_{i=1}^{m}\frac{\xi_{n,n+1}^i R_{g_i}^i(n)}{\p _{\tilde G}'(\tilde g_i)R_{g_i}^{i}(-1)},\quad n\ge 0,\\
\label{foeq2}
0&=\sum_{i=1}^{m}\frac{R_{g_i}^i(n)}{\p _{\tilde G}'(\tilde g_i)\xi_{-1,-n-1}^iR_{g_i}^{i}(-1)},\quad 1-m\le n<0,\\
\label{foeq3}
0&\not =\sum_{i=1}^{m}\frac{R_{g_i}^i(-m)}{\p _{\tilde G}'(\tilde g_i)\xi_{-1,m-1}^iR_{g_i}^{i}(-1)},
\end{align}
where $\tilde g_h$ are the $m$ different numbers (\ref{diffn}) and $\p_{\tilde G}(x)=\prod_{j=1}^m(x-\tilde g_i)$.

We show that in the following two Lemmas.

\begin{lemma}\label{lort1} Assume that (\ref{foeq1}), (\ref{foeq2}) and (\ref{foeq3}) hold.
For an integer $n$, $n\ge 1-m$ and a polynomial $s$ with $\deg(s)\le n+m-1$, we have
\begin{align}\label{lort1.1}
\langle \tilde \rho ,s(x)p_n(x)\rangle &=(-1)^n c_G\sum_{i=1}^{m}s(-\tilde g_i)\frac{\xi_{n,n+1}^i R_{g_i}^i(n)}{\p _{\tilde G}'(\tilde g_i)R_{g_i}^{i}(-1)},\quad \mbox{if $n\ge 0$,}\\\label{lort1.2}
0&=\sum_{i=1}^{m}s(-\tilde g_i)\frac{R_{g_i}^i(n)}{\p _{\tilde G}'(\tilde g_i)\xi_{-1,-n-1}^iR_{g_i}^{i}(-1)}, \quad \mbox{if $1-m\le n<0$.}
\end{align}
On the other hand, if $\deg(s)=n+m$,
\begin{align}\label{lort1.3}
\langle \tilde \rho ,s(x)p_n(x)\rangle & \not =(-1)^n c_G\sum_{i=1}^{m}s(-\tilde g_i)\frac{\xi_{n,n+1}^i R_{g_i}^i(n)}{\p _{\tilde G}'(\tilde g_i)R_{g_i}^{i}(-1)},\quad \mbox{if $n\ge 0$,}\\ \label{lort1.4}
0& \not =\sum_{i=1}^{m}s(-\tilde g_i)\frac{R_{g_i}^i(n)}{\p _{\tilde G}'(\tilde g_i)\xi_{-1,-n-1}^iR_{g_i}^{i}(-1)},\quad \mbox{if $n< 0$.}
\end{align}
\end{lemma}

\begin{proof}
It is enough to prove the Lemma for $s(x)=x^j$.

We prove (\ref{lort1.1}) and (\ref{lort1.2}) by induction on $j$.

For $j=0$, (\ref{lort1.1}) and (\ref{lort1.2}) are just (\ref{foeq1}) and (\ref{foeq2}).

Assume (\ref{lort1.1}) and (\ref{lort1.2}) hold for $s(x)=x^j$, $j\le n+m-1$, and consider $s(x)=x^{j+1}$, with $j+1\le n+m-1$. In particular
$j\le n+m, n+m-1,n+m-2$.
For $n>0$, using the induction hypothesis, (\ref{ttrr}) and (\ref{relaR}), we get
\begin{align*}
\langle \tilde\rho ,x^{j+1}p_n(x)\rangle&=\langle \tilde \rho,a_{n+1}x^{j}p_{n+1}(x)+b_nx^{j}p_n(x)+c_nx^{j}p_{n-1}(x)\rangle \\
&=(-1)^{n}c_G\sum_{i=1}^{m}\frac{(-\tilde g_i)^j}{\p _{\tilde G}'(\tilde g_i)R_{g_i}^i(-1)}\\
&\quad \quad \times (-\xi_{n+1,n+2}^ia_{n+1}R_{g_i}^i(n+1)+\xi_{n,n+1}^ib_nR_{g_i}^i(n)-\xi_{n-1,n}^ic_nR_{g_i}^i(n-1))\\
&=(-1)^{n}c_G\sum_{i=1}^{n_{\tilde G}}\frac{(-\tilde g_i)^j}{\p _{\tilde G}'(\tilde g_i)R_{g_i}^i(-1)}\xi_{n,n+1}^i(-\tilde g_i)R_{g_i}^i(n)\\
&=(-1)^{n}c_G\sum_{i=1}^{n_{\tilde G}}(-\tilde g_i)^{j+1}\frac{\xi_{n,n+1}^i}{\p _{\tilde G}'(\tilde g_i)R_{g_i}^i(-1)}R_{g_i}^i(n).
\end{align*}
For $n=0$ (hence $j\le m-2$), using the induction hypothesis, (\ref{ttrr}), (\ref{relaR}) and (\ref{sumuif}) we have
\begin{align*}
\langle \tilde \rho ,x^{j+1}p_0(x)\rangle&=\langle \tilde \rho ,a_{1}x^{j}p_{1}(x)+b_0x^{j}p_0(x)\rangle \\
&=c_G\sum_{i=1}^{m}\frac{(-\tilde g_i)^j}{\p _{\tilde G}'(\tilde g_i)R_{g_i}^i(-1)}(-\xi_{1,2}^ia_{1}R_{g_i}^i(1)+\xi_{0,1}^ib_0R_{g_i}^i(0))\\
&=c_G\sum_{i=1}^{m}\frac{(-\tilde g_i)^j}{\p _{\tilde G}'(\tilde g_i)R_{g_i}^i(-1)}(-\xi_{1,2}^ia_{1}R_{g_i}^i(1)+\xi_{0,1}^ib_0R_{g_i}^i(0)-\xi_{-1,0}^ic_0R_{g_i}^i(-1)\\
&\hspace{5cm}+\xi_{-1,0}^ic_0R_{g_i}^i(-1))\\
&=c_G\sum_{i=1}^{m}(-\tilde g_i)^{j+1}\frac{\xi_{0,1}^i}{\p _{\tilde G}'(\tilde g_i)R_{g_i}^i(-1)}R_{g_i}^i(0)+
c_Gc_0\sum_{i=1}^{m}\frac{(-\tilde g_i)^j}{\p _{\tilde G}'(\tilde g_i)}\\
&=c_G\sum_{i=1}^{m}(-\tilde g_i)^{j+1}\frac{\xi_{0,1}^i}{\p _{\tilde G}'(\tilde g_i)R_{g_i}^i(-1)}R_{g_i}^i(0).
\end{align*}
For $n=-1$ (hence $j+1\le m-2$), using  (\ref{sumuif}) we have
\begin{align*}
\sum_{i=1}^{m}&(-\tilde g_i)^{j+1}\frac{1}{\xi_{-1,0}^i\p _{\tilde G}'(\tilde g_i)R_{g_i}^i(-1)}R_{g_i}^i(-1) \\
&=\sum_{i=1}^{m}\frac{(-\tilde g_i)^{j+1}}{\p _{\tilde G}'(\tilde g_i)}=0.
\end{align*}
Finally, for $m-1\le n<-1$ using the induction hypothesis and (\ref{relaR}), we get
\begin{align*}
\sum_{i=1}^{m}&\frac{(-\tilde g_i)^{j+1}R_{g_i}^i(n)}{\xi_{-1,-n-1}^i\p _{\tilde G}'(\tilde g_i)R_{g_i}^i(-1)} \\
&=-\sum_{i=1}^{m}\frac{(-\tilde g_i)^j(\varepsilon_{n+1}^ia_{n+1}R_{g_i}^i(n+1)-b_nR_{g_i}^i(n)+\frac{c_n}{\varepsilon_n^i}R_{g_i}^i(n-1))}{\xi_{-1,-n-1}^i\p _{\tilde G}'(\tilde g_i)R_{g_i}^i(-1)}\\
&=-a_{n+1}\sum_{i=1}^{m}\frac{(-\tilde g_i)^jR_{g_i}^i(n+1)}{\xi_{-1,-n-2}^i\p _{\tilde G}'(\tilde g_i)R_{g_i}^i(-1)}+b_n\sum_{i=1}^{m}\frac{(-\tilde g_i)^jR_{g_i}^i(n)}{\xi_{-1,-n-1}^i\p _{\tilde G}'(\tilde g_i)R_{g_i}^i(-1)} \\ &\hspace{2 cm}-c_{n}\sum_{i=1}^{m}\frac{(-\tilde g_i)^jR_{g_i}^i(n-1)}{\xi_{-1,-n}^i\p _{\tilde G}'(\tilde g_i)R_{g_i}^i(-1)}
\\
&=0.
\end{align*}

(\ref{lort1.3}) and (\ref{lort1.4}) can be proved in a similar way by induction on $j$.

\end{proof}

\begin{lemma} \label{lort2}
Assume that (\ref{foeq1}), (\ref{foeq2}) and (\ref{foeq3}) hold. Then the polynomials
$q_n^G$, $n\ge 0$, are orthogonal with respect to $\tilde \rho$.
\end{lemma}

\begin{proof}

Write $s(x)=x^j$ and $\displaystyle\alpha_i^j=\frac{c_G(-\tilde g_i)^j}{\p _{\tilde G}'(\tilde g_i)R_{g_i}^i(-1)}$. According to Lemma \ref{lort1}  we have
\begin{align*}
\langle \tilde \rho ,x^jp_n(x)\rangle &= (-1)^n\sum_{i=1}^{m}\alpha_i^j\xi_{n,n+1}^iR_{g_i}^i(n),\quad \mbox{for $j\le n+m-1$,}\\
\langle \tilde \rho ,x^{n+m}p_n(x)\rangle &\not = (-1)^n\sum_{i=1}^{m}\alpha_i^{n+m}\xi_{n,n+1}^iR_{g_i}^i(n).
\end{align*}

It is easy to check that $\xi_{n-l,n-l+1}^i=\xi_{n-m,n-m+1}^i\xi_{n-l,m-l}^i$, for any $n-l\in \ZZ$. Hence
the previous identities can be rewritten in the form
\begin{align*}
\langle \tilde \rho ,x^jp_{n-l}(x)\rangle &
=(-1)^l\sum_{i=1}^{m}(-1)^{n}\xi_{n-m,n-m+1}^i\alpha_i^j\xi_{n-l,m-l}^iR_{g_i}^i(n-l),\\
\langle \tilde \rho ,x^{n-l+m}p_{n-l}(x)\rangle &
\not =(-1)^l\sum_{i=1}^{m}(-1)^{n}\xi_{n-m,n-m+1}^i\alpha_i^{n-l+m}\xi_{n-l,m-l}^iR_{g_i}^i(n-l),
\end{align*}
where  $j\le n-l+m-1$, $l=0,\ldots, m$ and $n-l\ge 0$. In particular, we have for $j\le n-1$, $l=0,\ldots, m$ and $n-l\ge 0$
\begin{align}\label{idc.1}
\langle \tilde \rho ,x^jp_{n-l}(x)\rangle &
=(-1)^l\sum_{i=1}^{m}(-1)^{n}\xi_{n-m,n-m+1}^i\alpha_i^j\xi_{n-l,m-l}^iR_{g_i}^i(n-l),\\\label{idc.2}
\langle \tilde \rho ,x^{n}p_{n-m}(x)\rangle &
\not =(-1)^m\sum_{i=1}^{m}(-1)^{n}\xi_{n-m,n-m+1}^i\alpha_i^{n}R_{g_i}^i(n-m).
\end{align}
In a similar way, we have
\begin{align}\label{idc.3}
0 &=(-1)^l\sum_{i=1}^{m}(-1)^{n}\xi_{n-m,n-m+1}^i\alpha_i^j\xi_{n-l,m-l}^iR_{g_i}(n-l),\\\label{idc.4}
0 &
\not =(-1)^m\sum_{i=1}^{m}(-1)^{n}\xi_{n-m,n-m+1}^i\alpha_i^nR_{g_i}(n-m),
\end{align}
where  $j\le n-1$, $l=0,\ldots, m$ and $n-l< 0$.

From (\ref{quso}) it follows that
$$
\langle \tilde \rho ,x^jq_n\rangle =\begin{vmatrix}
\langle \tilde \rho ,x^jp_n(x)\rangle  & -\langle \tilde \rho ,x^jp_{n-1}(x)\rangle  & \cdots & (-1)^m\langle \tilde \rho ,x^jp_{n-m}(x)\rangle  \\
\displaystyle\xi_{n,m}^1R_{g_1}^1(n) & \xi_{n-1,m-1}^1R_{g_1}^1(n-1) & \cdots &
R_{g_1}^1(n-m) \\
               \vdots & \vdots & \ddots & \vdots \\
              \xi_{n,m}^mR_{g_m}^m(n) & \displaystyle
               \xi_{n-1,m-1}^mR_{g_m}^m(n-1) & \cdots &R_{g_m}^m(n-m)
             \end{vmatrix}.
$$
We subtract to the first row of this determinant the linear combination of the next rows multiplying the $i$-row by $(-1)^{n}\xi_{n-m,n-m+1}^i\alpha_i^j$.
(\ref{idc.1}), (\ref{idc.2}), (\ref{idc.3}) and (\ref{idc.4}) then give
$$
\langle \tilde \rho ,x^jq_n\rangle =\begin{vmatrix}
0  & 0 & \cdots & 0  \\
\displaystyle\xi_{n,m}^1R_{g_1}^1(n) & \xi_{n-1,m-1}^1R_{g_1}^1(n-1) & \cdots &
R_{g_1}^1(n-m) \\
               \vdots & \vdots & \ddots & \vdots \\
              \xi_{n,m}^mR_{g_m}^m(n) & \displaystyle
               \xi_{n-1,m-1}^mR_{g_m}^m(n-1) & \cdots &R_{g_m}^m(n-m)
             \end{vmatrix}=0,
$$
for $0\le j\le n-1$, and
\begin{align*}
\langle \tilde \rho ,x^nq_n\rangle &=\begin{vmatrix}
0  & 0 & \cdots & d  \\
\displaystyle\xi_{n,m}^1R_{g_1}^1(n) & \xi_{n-1,m-1}^1R_{g_1}^1(n-1) & \cdots &
R_{g_1}^1(n-m) \\
               \vdots & \vdots & \ddots & \vdots \\
              \xi_{n,m}^mR_{g_m}^m(n) & \displaystyle
               \xi_{n-1,m-1}^mR_{g_m}^m(n-1) & \cdots &R_{g_m}^m(n-m)
             \end{vmatrix}\\&=(-1)^{m}d\left(\prod_{h=1}^m\varepsilon_{n-m+1}^h\right)\Omega_G(n+1),
\end{align*}
for certain number $d\not =0$. Hence $\langle \tilde \rho ,x^nq_n\rangle\not =0$.

This proves that the polynomials $(q_n)_n$ are orthogonal with respect to the measure $\tilde \rho$.

\end{proof}

\bigskip
We still need a  last ingredient for identifying the measure $\tilde \rho$ with respect to which the polynomials $(q_n^G)_n$ (\ref{quso}) are orthogonal.
The measures in the conjectures A, B and C  depend on certain finite sets $F$ while the polynomials $(q_n^G)_n$ depend on the finite set $G$ (the degrees of the polynomials $R$'s). The relationship between the sets $F$ and $G$ are going to be given by the following transforms of finite sets of positive integers.

Consider the sets $\Upsilon$ and $\Upsilon _0$ formed by all finite sets of positive or nonnegative integers, respectively:
\begin{align*}
\Upsilon&=\{F:\mbox{$F$ is a finite set of positive integers}\} ,\\
\Upsilon_0&=\{F:\mbox{$F$ is a finite set of nonnegative integers}\} .
\end{align*}
We consider an involution $I$ in $\Upsilon$, and a family $J_h$, $h\ge 1$, of transforms from  $\Upsilon$ into  $\Upsilon _0$. For $F\in \Upsilon$ write $F=\{f_1,\ldots ,f_k\}$ with $f_i<f_{i+1}$, so that $f_k=\max F$. Then $I(F)$ and $J_h(F)$, $h\ge 1$, are defined by
\begin{align}\label{dinv}
I(F)&=\{1,2,\ldots, f_k\}\setminus \{f_k-f,f\in F\},\\ \label{dinv2}
J_h(F)&=\{0,1,2,\ldots, f_k+h-1\}\setminus \{f-1,f\in F\}.
\end{align}
Notice that $I$  is an involution: $I^2=Id$.

For the involution $I$, the bigger the wholes in $F$ (with respect to the set $\{1,2,\ldots , f_k\}$), the bigger the involuted set $I(F)$.
Here it is a couple of examples
$$
I(\{ 1,2,3,\ldots ,k\})=\{ k\},\quad \quad I(\{1, k\})=\{ 1,2,\ldots, k-2, k\}.
$$
Something similar happens for the transform $J_h$ with respect to $\{0,1,\ldots , f_k+h-1\}$.

Notice that
$$
\max F=\max I(F), \quad h-1+\max F=\max J_h(F),
$$
and if we write $n_F$ for the number of elements in $F$, we also have
$$
n_{I(F)}=f_k-n_F+1,\quad n_{J_h(F)}=f_k+h-n_F.
$$

\section{Charlier polynomials}\label{sch}

In this Section we consider the Charlier family and prove Conjecture A.
For $a\neq0$, we write $(c_n^a)_n$ for the sequence of Charlier polynomials (that and the next formulas can be found in \cite{Ch}, pp. 170-1; see also \cite{KLS}, pp, 247-9 or \cite{NSU}, ch. 2) defined by
\begin{equation}\label{Chpol}
    c_n^a(x)=\frac{1}{n!}\sum_{j=0}^n(-a)^{n-j}\binom{n}{j}\binom{x}{j}j!.
\end{equation}
The Charlier polynomials are orthogonal with respect to the measure
\begin{equation}\label{Chw}
    \rho_a=\sum_{x=0}^{\infty}\frac{a^x}{x!}\delta_x,\quad a\neq0,
\end{equation}
which is positive only when $a>0$. The three-term recurrence formula for $(c_n^a)_n$ is ($c_{-1}^a=0$)
\begin{equation}\label{Chttrr}
   xc_n^a=(n+1)c_{n+1}^a+(n+a)c_n^a+ac_{n-1}^a,\quad n\geq0.
\end{equation}
They are eigenfunctions of the following second-order difference operator
\begin{equation}\label{Chdeq}
   D_a=-x\mathfrak{s}_{-1}+(x+a)\mathfrak{s}_0-a\mathfrak{s}_1,\quad D_a(c_n^a)=nc_n^a,\quad n\geq0,
\end{equation}
where $\mathfrak{s}_j(f)=f(x+j)$. They also satisfy
\begin{equation}\label{Chlad}
   \Delta(c_n^a)=c_{n-1}^a,\quad \frac{d}{da}(c_n^a)=-c_{n-1}^a,
\end{equation}
and the duality
\begin{equation}\label{Chdua}
   (-1)^ma^mn!c_n^a(m)=(-1)^na^nm!c_{m}^a(n), \quad n,m\ge 0.
\end{equation}

Consider now the algebra $\A$ of operators defined by \eqref{algdiffd}. Lemma 4.1 of \cite{du1} provides the following $\D$-operator for Charlier polynomials: for $a\neq0$, the operator $\mathcal{D}$ defined by \eqref{defTo} from the sequence $\varepsilon_n=1, n\geq0$, is a $\mathcal{D}$-operator for the  Charlier polynomials $(c_n^a)_n$ \eqref{Chpol} and the algebra $\A$. More precisely
$$
\mathcal{D}=\nabla.
$$
Since we only have one $\D$-operator for Charlier polynomials, we take $\varepsilon _n^h=\varepsilon _n=1$ in Theorem \ref{Teor1}, and when we use the results in the previous Section, we will remove all the superindices. In particular, the auxiliary numbers $\xi_{x,i}^h=\xi_{x,i}$ (see (\ref{defxi})) are then
\begin{equation}\label{defxich}
\xi_{x,i}=1.
\end{equation}

We can apply Theorem \ref{Teor1} to produce from arbitrary polynomials $R_j$, $j=1,\ldots , m$, a large class of sequences of polynomials $(q_n)_n$ satisfying higher order difference equations. But only for a convenient choice of the polynomials $R_j$, $j\ge 0$, these polynomials $(q_n)_n$ are also orthogonal with respect to a measure.
For $m=1$, it has been shown in \cite{du1} that $R_1(x)$ can be chosen to be equal to $c_j^{-a}(-x-1)$, $j\ge 0$. We now prove that, indeed,
the  polynomials $c_j^{-a}(-x-1)$, $j\ge 0$, satisfy the recurrence (\ref{relaR0}) for $\eta=\kappa=1$.

\begin{lemma}\label{lemRch} Consider the polynomials $R_j(x)=c_j^{-a}(-x-1)$, where $(c_j^{-a})_j$ are Charlier polynomials (\ref{Chpol}). Then they satisfy the recurrence (\ref{relaR0}), where $\varepsilon_n=1$, $\eta=\kappa=1$ and $a_n$, $b_n$ and $c_n$ are the recurrence sequences for the Charlier polynomials $(c_n^a)_n$ (\ref{Chttrr}).
\end{lemma}

\begin{proof}
We have to prove that
$$
(n+1)c_j^{-a}(-n-2)-(n+a)c_j^{-a}(-n-1)+ac_j^{-a}(-n)=(j+1)c_j^{-a}(-n-1),\quad n\in \ZZ.
$$
But this follows straightforwardly by writing $x=-n-1$ in the second order difference equation (\ref{Chdeq}) for the Charlier polynomials $(c_j^{-a})_j$.
\end{proof}

We are now ready to prove the Theorem \ref{mthch} in the Introduction.

\begin{proof}
To prove (1) of Theorem \ref{mthch}, we use the strategy of the previous Section. Observe first that from the previous Lemma we have $\tilde G=\{g_1+1,\ldots,g_m+1\}$.

To start with, we assume the following claim (which we prove later).
\begin{align}\label{foeq1ch}
\langle \tilde \rho_a^F,c_n^a\rangle &=e^aa^{g_m}(-1)^{n+m-1}\sum_{i=1}^{m}\frac{c_{g_i}^{-a}(-n-1)}{\p _{\tilde G}'(g_i+1)c_{g_i}^{-a}(0)},\quad n\ge 0,\\
\label{foeq2ch}
0&=\sum_{i=1}^{m}\frac{c_{g_i}^{-a}(-n-1)}{\p _{\tilde G}'(g_i+1)c_{g_i}^{-a}(0)},\quad 1-m\le n<0,\\
\label{foeq3ch}
0&\not =\sum_{i=1}^{m}\frac{c_{g_i}^{-a}(m-1)}{\p _{\tilde G}'(g_i+1)c_{g_i}^{-a}(0)},
\end{align}
where $\p _{\tilde G}(x)=\prod_{j=1}^m(x-g_i-1)$.

Since $R_{g_i}(x)=c_{g_i}^{-a}(-x-1)$ (Lemma \ref{lemRch}) and $\xi _{x,y}=1$, for all $x\in \RR$, $y\ge 0$ (\ref{defxich}), the orthogonality of the polynomials $(q_n)_n$ with respect to $\tilde \rho_a^F$ is now a consequence of Lemma \ref{lort2}.

We now prove (2) of Theorem \ref{mthch}. The fact that the polynomials $(q_n)_n$ are eigenfunctions of a higher order difference operator is a direct consequence of Theorem \ref{Teor1}. Indeed, it is enough to take into account that the $\D $-operator for the Charlier polynomials is $\mathcal{D}=\nabla$ and it is generated by the sequence $\varepsilon _n=1$, $n\ge 1$. Take now $S=1$  in Theorem \ref{Teor1} and consider the difference operator $D_q$ given by (\ref{Dq}).
Write $F=\{f_1,\ldots ,f_k\}$ with $f_i<f_{i+1}$, so that $f_k=\max F$ and $k$ is the number of elements of $F$.
We now prove that the difference operator $D_q$ has the form (\ref{doho}) with $-s=r=\sum_{f\in F}f-\frac{k(k-1)}{2}+1$.

Since $G=\{g_1,\ldots , g_m\}$ with $g_i<g_{i+1}$, using Lemma \ref{lgp1} we conclude that $\Omega _G$ is a polynomial of degree $\sum_{g\in G}g-m(m-1)/2$.
Hence the polynomial $P$ defined by $P(x)-P(x-1)=\Omega (x)$ has degree $\sum_{g\in G}g-m(m-1)/2+1$.
Taking into account that $G=I(F)$ and the definition of the involution $I$ (\ref{dinv}), we have
$$
\sum_{g\in G}g=f_k(f_k+1)/2-kf_k+\sum_{f\in F}f.
$$
Since $m=n_{I(F)}=f_k-k+1$, we get
$$
\sum_{g\in G}g-\frac{m(m-1)}{2}+1=\sum_{f\in F}f-\frac{k(k-1)}{2}+1=r.
$$
That is, $P$ is a polynomial of degree $r$. This implies that the operator $P(D_a)$ has the form
$$
\sum _{l=-r}^{r}\tilde h_{l}(x)\Sh_l,
$$
where $\tilde h_r,\tilde h_{-r}\not =0$. Moreover, taking into account that the coefficient of $\mathfrak{s}_{-1}$  in $D_a$ \eqref{Chdeq} is $-x$, we have $\tilde h_{-r}(x)=u_1(-1)^{r}x(x-1)\cdots (x-r+1)$, where $u_1$ denotes the leading coefficient of the polynomial $P$. Hence $\deg(\tilde h_{-r}(x))=r$.

On the other hand, using Lemma \ref{lgp2}, we can conclude that the polynomials $M_h$ (\ref{emeiexp}) have degree at most
$v_h=\sum_{g\in G}g-m(m-1)/2-g_h$. Since $R_h$ has degree $g_h$, this shows that the operator $M_h(D_a)\nabla R_h(D_a)$ has the form
$$
\sum _{l=-r}^{r-1}\hat h_{l}(x)\Sh_l,
$$
where $\hat h_{-r}(x)=u_2u_3(-1)^{r-1}x(x-1)\cdots (x-v_h+1)(x-v_h-1)\cdots (x-r+1)$ where $u_2$ is the leading coefficient of $R_h$ and
$u_3$ is the coefficient of $x^{v_h}$ in $M_h$. Hence $\deg(\hat h_{-r}(x))\le r-1$.
To complete the proof of (2), it is enough to take into account the expression of $D_q$ given by (\ref{Dq}).

\bigskip

We finally prove the claim. We proceed in two steps.

\bigskip

\noindent
\textbf{First step}.
\textit{Assume that (\ref{foeq1ch}), (\ref{foeq2ch}) and (\ref{foeq3ch}) hold for the singleton $F=\{l\}$, $l\ge 1$. Then
(\ref{foeq1ch}), (\ref{foeq2ch}) and (\ref{foeq3ch}) also hold for any finite set $F$ of positive integers.}

Write $F=\{ f_1,f_2,\ldots , f_k\}$, with $f_i< f_{i+1}$, and $F_1=\{ f_2,\ldots , f_k\}$. We first prove that if (\ref{foeq1ch}), (\ref{foeq2ch}) and (\ref{foeq3ch}) hold for $F_1$, then also hold for $F$.

We have that $\tilde \rho_a^F=(x+f_k+1-f_1)\tilde \rho _a^{F_1}$. It is not difficult to see that if $I(F_1)=G_1=\{ g_1^1,g_2^1,\ldots , g_{m+1}^1\}$ then $I(F)=G=\{ g_1^1,g_2^1,\ldots , g_{m+1}^1\}\setminus \{f_k-f_1\} $; write then $g_{i_0}^1=f_k-f_1$ and $G=\{ g_1,g_2,\ldots , g_{m}\}$. Call also $\tilde G_1=\{ g_1^1+1,g_2^1+1,\ldots , g_{m+1}^1+1\}$ and $\tilde G=\{ g_1+1,g_2+1,\ldots , g_{m}+1\}$.This gives for $i=1,\ldots ,i_0-1$ that $g_i=g_i^{1}$ and
$$
\frac{1}{\p _{\tilde G}'(g_i+1)}=\frac{g_i-f_k+f_1}{\p _{\tilde G_1}'(g_i^1+1)},
$$
and for $i=i_0,\ldots ,m$, $g_i=g_{i+1}^{1}$ and
$$
\frac{1}{\p _{\tilde G}'(g_i+1)}=\frac{g_i-f_k+f_1}{\p _{\tilde G_1}'(g_{i+1}^1+1)}.
$$
Using Lemma \ref{lort1}, we have for $n\ge 0$
\begin{align*}
\langle \tilde \rho_a^F,p_n\rangle &=\langle \tilde \rho_a^{F_1},(x+f_k+1-f_1)p_n\rangle \\
&=e^aa^{g^1_{m+1}}(-1)^{n+m}\sum_{i=1}^{m+1}(-g_i^1+f_k-f_1)\frac{c_{g_i}^{-a}(-n-1)}{\p _{\tilde G_1}'(g_i^1+1)c_{g_i}^{-a}(0)}\\
&=e^aa^{g_m}(-1)^{n+m-1}\sum_{i=1}^{m}\frac{c_{g_i}^{-a}(-n-1)}{\p _{\tilde G}'(g_i+1)c_{g_i}^{-a}(0)}.
\end{align*}
On the other hand for $1-m\le n<0$ we have
$$
\sum_{i=1}^{m}\frac{c_{g_i}^{-a}(-n-1)}{\p _{\tilde G}'(g_i+1)c_{g_i}^{-a}(0)}
=\sum_{i=1}^{m+1}(g_i^1-f_k+f_1)\frac{c_{g_i}^{-a}(-n-1)}{\p _{\tilde G_1}'(g_i^1+1)c_{g_i}^{-a}(0)},
$$
since $1\le n+m$, by applying (\ref{lort1.2}) to $G_1$, we get that this last sum is equal to $0$
(take into account that $R_{g_i}(x)=c_{g_i}^{-a}(-x-1)$  and $\xi _{x,y}=1$, for all $x\in \RR$, $y\ge 0$).
Finally for $n=-m$, by applying (\ref{lort1.4}) to $G_1$, we get
$$
\sum_{i=1}^{m}\frac{c_{g_i}^{-a}(m-1)}{\p _{\tilde G}'(g_i+1)c_{g_i}^{-a}(0)}
=\sum_{i=1}^{m+1}(g_i^1-f_k+f_1)\frac{c_{g_i}^{-a}(m-1)}{\p _{\tilde G_1}'(g_i^1+1)c_{g_i}^{-a}(0)}\not =0.
$$
The proof of the Step 1 now follows easily by induction.

\bigskip

\noindent
\textbf{Second step}.
\textit{Let $m$ be a positive integer. Then
(\ref{foeq1ch}), (\ref{foeq2ch}) and (\ref{foeq3ch})  hold for $F=\{ m\}$.}

Indeed, since $I(F)=G=\{1,2,\ldots ,m\}$, write $\tilde G=\{2,3,\ldots ,m+1\}$. Then we have that $\p_{\tilde G}(x)=m!\binom{x-2}{m}$, and hence
$$
\frac{1}{\p_{\tilde G}'(i+1)}=\frac{(-1)^{m+i}}{(m-1)!}\binom{m-1}{i-1}.
$$
Since $c_i^{-a}(0)=a^i/i!$, we then have to prove that
\begin{align}\label{lortCh1}
\langle \tilde\rho _a^F, c_n^a\rangle&=(-1)^n e^a\sum_{i=1}^m\frac{(-1)^{i+1}}{(m-1)!}\binom{m-1}{i-1}i!a^{m-i}c_i^{-a}(-n-1),\quad n\ge0\\ \label{lortCh2}
0&=\sum_{i=1}^m\frac{(-1)^{i+1}}{(m-1)!}\binom{m-1}{i-1}i!a^{m-i}c_i^{-a}(-n-1),\quad 1-m\le n< 0,
\\ \label{lortCh3}
0&\not =\sum_{i=1}^m\frac{(-1)^{i+1}}{(m-1)!}\binom{m-1}{i-1}i!a^{m-i}c_i^{-a}(m-1).
\end{align}
We first prove (\ref{lortCh2}) and (\ref{lortCh3}).
For $n<0$, using the duality (\ref{Chdua}) we get $i!c_i^{-a}(-n-1)=(-n-1)!a^{i+n+1}c_{-n-1}^{-a}(i)$. Hence
\begin{align*}
(-1)^{m+1}\sum_{i=1}^m&\frac{i!a^{m-i}c_i^{-a}(-n-1)}{\p_{\tilde G}'(i+1)}\\&=
(-1)^{m+1}(-n-1)!\sum_{i=1}^m\frac{a^{m-i}a^{i+n+1}c_{-n-1}^{-a}(i)}{\p_{\tilde G}'(i+1)}
\\&=
(-1)^{m+1}a^{m+n+1}\sum_{i=1}^m\frac{c_{-n-1}^{-a}(i)}{\p_{\tilde G}'(i+1)}.
\end{align*}
Since $c_{-n-1}^{-a}(i)$ is a polynomial in $i$ of degree $-n-1$, we deduce (\ref{lortCh2}) and (\ref{lortCh3}) from (\ref{sumuif}).

We now prove (\ref{lortCh1}). Write
$$
\nu_{m,n}(a)=(-1)^n e^a\sum_{i=1}^m\frac{(-1)^{i+1}}{(m-1)!}\binom{m-1}{i-1}i!a^{m-i}c_i^{-a}(-n-1).
$$
To stress the dependence of $\tilde \rho_a^F$ on $m$ and $a$, we write $\tilde \rho_a^m$ instead of $\tilde \rho_a^F$. We have defined $\tilde \rho_a^m$ and $(c_n^a)_n$ only for $a\not =0$, but it is easy to see that the limit case $a=0$ is given by $\tilde \rho_0^m=-m\delta_{-m-1}$ and $c_n^0(x)=\binom{x}{n}$. A simple computation shows that $\langle\tilde \rho_0^m,c_n^0\rangle =v_{m,n}(0)$.
Hence, since both $\nu_{n,m}(a)$ and $\langle\tilde \rho_a^m,c_n^a\rangle$ are analytic functions of $a$, (\ref{lortCh1}) will follow if we prove that
\begin{equation}\label{xx1}
\frac{d}{da}(\langle \tilde \rho _a^m, c_n^a\rangle) =\frac{d}{da}(\nu_{m,n}(a)),\quad n\ge 0.
\end{equation}

We proceed by induction on $m$. For $m=1$ and $n\ge 0$ we have from \cite{du1}, Lemma 4.2 that
$$
\langle \tilde \rho _a^1, c_n^a\rangle=(-1)^ne^ac_1^{-a}(-n-1)=\nu_{1,n}(a),
$$
from where we see that (\ref{xx1}) holds for $m=1$.

It is easy to see that $d(c_n^a)/da=-c_{n-1}^a$ and $d(\tilde \rho_{a,x}^m)/da=\tilde \rho_{a,x}^{m-1}$, where $\tilde\rho_{a,x}^m$ is the mass
of the measure $\tilde\rho_{a}^m$ at $x$.
Hence
\begin{equation*}\label{xx2}
\frac{d}{da}(\langle \tilde \rho _a^m, c_n^a\rangle) =\langle \tilde \rho _a^{m-1}, c_n^a\rangle-\langle \tilde \rho _a^m, c_{n-1}^a\rangle .
\end{equation*}
On the other hand, using (\ref{Chlad}), we get
\begin{align*}
\frac{d}{da}(\nu_{m,n}(a))&=(-1)^ne^a\sum_{i=1}^m\frac{(-1)^{i+1}}{(m-1)!}
\binom{m-1}{i-1}i! \\&\quad \quad\quad\times (a^{m-i}(c_i^{-a}(-n-1)+c_{i-1}^{-a}(-n-1))+(m-i)a^{m-i-1}c_i^{-a}(-n-1))\\
&=(-1)^ne^a\sum_{i=1}^{m-1}\frac{(-1)^{i+1}}{(m-2)!} \binom{m-2}{i-1}i!a^{m-i-1}c_i^{-a}(-n-1)\\&\quad\quad\quad -(-1)^{n-1}e^a\sum_{i=1}^m\frac{(-1)^{i+1}}{(m-1)!}
\binom{m-1}{i-1}a^{m-i}i!c_i^{-a}(-n)\\
&=\nu_{m-1,n}(a)-\nu_{m,n-1}(a).
\end{align*}
In particular using (\ref{lortCh2}), we get for $n=0$ that $\frac{d}{da}(\nu_{m,0}(a))=\nu_{m-1,0}(a)$.

The proof of the Step 2 can now be completed easily by induction on $n$.

\bigskip

The proof of the Claim is now an easy consequence of Steps 1 and 2.

\end{proof}

\begin{corollary}\label{jodch}
Let $F$ be a finite set of positive integers and consider the associated set $I(F)=G=\{ g_1,\ldots, g_m\}$ where the involution $I$ is defined by (\ref{dinv}). Let $\rho_a$ be the Charlier measure (\ref{Chw}) and $(c_n^a)_n$ its sequence of orthogonal polynomials defined by (\ref{Chpol}). Assume that the $m\times m$ Casorati determinant $\Omega _F$ defined by
$$
\Omega_F (n)=\det \left(c^{-a}_{g_l}(-n-j-1)\right)_{l,j=1}^m,
$$
does not vanish for $n\ge 0$: $\Omega_F (n)\not =0$ for $n\ge 0$.
Then the measure $\rho _a^F$ has associated a sequence of orthogonal polynomials and Conjecture A is true.
\end{corollary}

\begin{proof}
The proof is an easy consequence of the previous Theorem, taking into account that the measure $\rho_a^F$ is
equal to the measure $\tilde \rho _a^F$ shifted by $f_k+1$ (with $f_k=\max F$):
$\rho_a^F=a^{f_k+1}\tilde \rho _a^F (x-f_k-1)$.
\end{proof}

\section{Meixner polynomials}\label{sme}

In this Section we apply our method to the Meixner family.
For $a\not =0, 1$ we write $(m_{n}^{a,c})_n$ for the sequence of Meixner polynomials defined by
\begin{equation}\label{Mxpol}
m_{n}^{a,c}(x)=\frac{a^n}{(1-a)^n}\sum _{j=0}^n a^{-j}\binom{x}{j}\binom{-x-c}{n-j}
\end{equation}
(we have taken a slightly different normalization from the one used in \cite{Ch}, pp. 175-7, from where
the next formulas can be easily derived; see also \cite{KLS}, pp, 234-7 or \cite{NSU}, ch. 2).
Meixner polynomials are eigenfunctions of the following second order difference operator
\begin{equation*}\label{Mxdeq}
D_{a,c} =\frac{x\Sh_{-1}-[(1+a)x+ac]\Sh_0+a(x+c)\Sh_1}{a-1},\qquad D_{a,c} (m_{n}^{a,c})=nm_{n}^{a,c},\quad n\ge 0.
\end{equation*}
When $a\not =0, 1$, they satisfy the following three term recurrence formula ($m_{-1}=0$)
\begin{equation}\label{Mxttrr}
xm_n=(n+1)m_{n+1}-\frac{(a+1)n+ac}{a-1}m_n+\frac{a(n+c-1)}{(a-1)^2}m_{n-1}, \quad n\ge 0
\end{equation}
(to simplify the notation we remove the parameters in some formulas).
Hence, for $a\not =0,1$ and $c\not =0,-1,-2,\ldots $, they are always orthogonal with respect to a moment functional $\rho_{a,c}$, which we
normalize
by taking $\langle \rho_{a,c},1\rangle =\Gamma(c)$. For $0<\vert a\vert<1$ and $c\not =0,-1,-2,\ldots $, we have
\begin{equation*}\label{MXw}
\rho_{a,c}=(1-a)^c\sum _{x=0}^\infty \frac{a^{x}\Gamma(x+c)}{x!}\delta _x.
\end{equation*}
The moment functional $\rho_{a,c}$ can be represented by a positive measure only when $0<a<1$ and $c>0$.

Meixner polynomials also satisfy
\begin{equation*}\label{Mxlad}
   \Delta(m_n^{a,c})=m_{n-1}^{a,c+1},\quad \frac{d}{da}(m_n^{a,c})=-\frac{c+n-1}{(a-1)^2}m_{n-1}^{a,c}.
\end{equation*}

\bigskip

Consider now the algebra of operators $\A$ defined by \eqref{algdiffd}. Lemma 5.1 of \cite{du1} provides two $\D$-operators for Meixner polynomials
(take into account that we are using here a different normalization for the Meixner polynomials):
For $a\neq0,1$, the operators $\mathcal{D}_1$ and $\mathcal{D}_2$ defined by \eqref{defTo} from the sequences $\varepsilon_n=a/(1-a)$ and $\tilde \varepsilon_n=1/(1-a), n\geq0$, respectively, are $\mathcal{D}$-operators for the  Meixner polynomials $(m_n^{a,c})_n$ \eqref{Mxpol} and the algebra $\A$. More precisely
$$
\mathcal{D}_1=\frac{a}{1-a}\Delta,\quad \mathcal{D}_2=\frac{1}{1-a}\nabla.
$$
Since we  have two $\D$-operators for Meixner polynomials, we make a partition of the  indices in Theorem \ref{Teor1}
and take $\varepsilon _n^h=a/(1-a)$, for $h=1,\ldots, m_1$ and $\varepsilon _n^h=1/(1-a)$, for $h=m_1+1,\ldots , m$. In particular, the auxiliary sequences of numbers $\xi_{x,i}^h$ (see (\ref{defxi})) are then
\begin{align}\label{defxime1}
\xi_{x,i}^h&=\frac{a^i}{(1-a)^i}, \quad \mbox{for $h=1,\ldots, m_1$,}\\
\nonumber\xi_{x,i}^h&=\frac{1}{(1-a)^i}, \quad \mbox{for $h=m_1+1,\ldots, m$.}
\end{align}

We can apply Theorem \ref{Teor1} to produce from arbitrary polynomials $R_j$, $j\ge 0$, a large class of sequences of polynomials $(q_n)_n$ satisfying higher order difference equations. But only for a convenient choice of the polynomials $R_j$, $j\ge 0$, these polynomials $(q_n)_n$ are also orthogonal with respect to a measure. There is again a very nice symmetry  in the choice of these polynomials. Indeed, the polynomials $R_j$ have to be Meixner polynomials but with different parameters.

\begin{lemma}\label{lemRme} Let $a\not =0,1$. Consider the polynomials
$$
R_j(x)=m_j^{1/a,2-c}(-x-1),\quad j\ge 0.
$$
Then they satisfy the recurrence (\ref{relaR0}), where $\varepsilon_n=a/(1-a)$, $\eta=-1$, $\kappa=c-1$ and $a_n$, $b_n$ and $c_n$ are the recurrence sequences for the Meixner polynomials $(m_n^{a,c})_n$ (\ref{Mxttrr}). In the same way, the polynomials
$$
R_j(x)=m_j^{a,2-c}(-x-1),\quad j\ge 0,
$$
satisfy the recurrence (\ref{relaR0}), where $\varepsilon_n=1/(1-a)$, $\eta=1$, $\kappa=1$  and $a_n$, $b_n$ and $c_n$ are the recurrence sequences for the Meixner polynomials $(m_n^{a,c})_n$ (\ref{Mxttrr}).
\end{lemma}

\begin{proof}
It is similar to the proof of Lemma \ref{lemRch} for the Charlier case.
\end{proof}

We are now ready to prove the main Theorem of this Section.

\begin{theorem} Let $F_1$ and $F_2$ be two finite sets of positive integers (the empty set is allowed, in which case we take $\max F=-1$). For $h\ge 1$, consider the transformed sets $J_h(F_1)=H=\{ h_{1},\ldots, h_{m_1}\}$ and $I(F_2)=K=\{ k_{1},\ldots, k_{m_2}\}$, where the involution $I$ and the transform $J_h$ are defined by (\ref{dinv}) and (\ref{dinv2}), respectively. Define $m=m_1+m_2$. Let $a$ and $c$ be real numbers satisfying $a\not =0,1$ and
$c\not \in \{f_{1,M}+f_{2,M}+h-l,l\ge 0\} $, where $f_{i,M}=\max F_i$. Consider the Meixner polynomials $(m_n^{a,c})_n$ (\ref{Mxpol}). Assume that $\Omega_{a,c}^{H,K} (n)\not =0$ for $n\ge 0$ where the $m\times m$ Casorati determinant $\Omega _{a,c}^{H,K}$ is defined by
\begin{equation*}\label{casortme}
\Omega_{a,c}^{H,K} (n)=\begin{vmatrix}
m^{1/a,2-c}_{h_{1}}(-n) & \cdots &
m^{1/a,2-c}_{h_{1}}(-n+m-1) \\
              \vdots & \ddots & \vdots \\
               \displaystyle
               m^{1/a,2-c}_{h_{m_1}}(-n) & \cdots &m^{1/a,2-c}_{h_{m_1}}(-n+m-1)
               \\ \vspace{-.3cm} \\
\displaystyle\frac{m^{a,2-c}_{k_{1}}(-n)}{a^{m-1}} & \cdots &
m^{a,2-c}_{k_{1}}(-n+m-1) \\
               \vdots & \ddots & \vdots \\
\displaystyle\frac{m^{a,2-c}_{k_{m_2}}(-n)}{a^{m-1}}
               & \cdots &m^{a,2-c}_{k_{m_2}}(-n+m-1)
             \end{vmatrix}.
\end{equation*}
We then define the sequence of polynomials $(q_n)_n$ by
\begin{equation*}\label{qusch}
q_n(x)=\begin{vmatrix}
\displaystyle\frac{(1-a)^m m^{a,c}_n(x)}{a^m} & \displaystyle\frac{-(1-a)^{m-1}m^{a,c}_{n-1}(x)}{a^{m-1}} & \cdots & (-1)^mm^{a,c}_{n-m}(x) \\
m^{1/a,2-c}_{h_{1}}(-n-1) & m^{1/a,2-c}_{h_{1}}(-n) & \cdots &
m^{1/a,2-c}_{h_{1}}(-n+m-1) \\
               \vdots & \vdots & \ddots & \vdots \\
               m^{1/a,2-c}_{h_{m_1}}(-n-1) & \displaystyle
               m^{1/a,2-c}_{h_{m_1}}(-n) & \cdots &m^{1/a,2-c}_{h_{m_1}}(-n+m-1)
               \\ \\
\displaystyle\frac{m^{a,2-c}_{k_{1}}(-n-1)}{a^m} & \displaystyle\frac{m^{a,2-c}_{k_{1}}(-n)}{a^{m-1}} & \cdots &
m^{a,2-c}_{k_{1}}(-n+m-1) \\
               \vdots & \vdots & \ddots & \vdots \\
\displaystyle \frac{m^{a,2-c}_{k_{m_2}}(-n-1)}{a^m} & \displaystyle\frac{m^{a,2-c}_{k_{m_2}}(-n)}{a^{m-1}}
               & \cdots &m^{a,2-c}_{k_{m_2}}(-n+m-1)
             \end{vmatrix}.
\end{equation*}
Then

\noindent
(1) The polynomials $(q_n)_n$ are orthogonal with respect to the measure
$$
\tilde \rho^{F_1,F_2,h}_{a,c}=\prod_{f\in F_1}(x+c-f)\prod_{f\in F_2}(x+f_{2,M}+1-f)\rho_{a,c-f_{1,M}-f_{2,M}-h-1}(x+f_{2,M}+1).
$$
(2) The polynomials $(q_n)_n$ are eigenfunctions of a higher order difference operator of the form (\ref{doho}) with
$$
-s=r=\sum _{f\in F_2}f-\sum _{f\in F_1}f-\frac{n_{F_1}(n_{F_1}-1)}{2}-\frac{n_{F_2}(n_{F_2}-1)}{2}+n_{F_1}(f_{1,M}+h)+1
$$
(which can be explicitly constructed using Theorem \ref{Teor1}).
\end{theorem}

\begin{proof}
First of all, notice that we have performed a straightforward normalization of the polynomials $q_n$, $n\ge 0$ (with respect to (\ref{qus})).

Write $R^i_j$, $i=1,\ldots ,m, j\ge 0$, for the polynomials
$$
R^i_j(x)=\begin{cases} m_j^{1/a,2-c}(-x-1),&i=1,\ldots , m_1,\\
m_j^{a,2-c}(-x-1),&i=m_1+1,\ldots , m,\end{cases}
$$
$G$ for the $m$-tuple $G=(h_1,\cdots,h_{m_1},k_1,\cdots,k_{m_2})=(g_1,\ldots, g_m)$
and $\p _{\tilde G}$ for the polynomial
$$
\p_{\tilde G}(x)=\prod_{j=1}^{m_1}(x+h_{j}-c+1)\prod_{j=m_1+1}^{m}(x-k_{j-m_1}-1),
$$
where $\tilde G=\{-h_1+c-1,\ldots,-h_{m_1}+c-1, k_1+1,\ldots, k_{m_2}+1\}$.

Since $c\not \in \{f_{1,M}+f_{2,M}+h-l+1,l\ge 0\} $, from the definition of $J_h(F_1)$ and $I(F_2)$ (see (\ref{dinv2}) and (\ref{dinv})), we conclude that $\p_{\tilde G}$ has simple roots.

The Theorem can be proved in a similar way to Theorem \ref{mthch} but using the following identities instead of (\ref{foeq1ch}), (\ref{foeq2ch}) and
(\ref{foeq3ch}):
\begin{align*}
\frac{(-1)^{m+1}(a-1)^{f_{1,M}+f_{2,M}+h-1}}{a^{f_{2,M}}\Gamma(c-1)}\langle \tilde \rho_{a,c}^{F_1,F_2,h},m_n^{a,c}\rangle &=(-1)^n\sum_{i=1}^{m}\frac{\xi^i_{n,n+1}R^i_{g_i}(n)}{\p _{\tilde G}'(\tilde g_i)R^i_{g_i}(-1)},\quad n\ge 0,\\
\sum_{i=1}^{m}\frac{R^i_{g_i}(n)}{\xi^i_{-1,-n-1}\p _{\tilde G}'(\tilde g_i)R^i_{g_i}(-1)}&=0,\quad 1-m\le n<0,\\
\sum_{i=1}^{m}\frac{R^i_{g_i}(n)}{\xi^i_{-1,m-1}\p _{\tilde G}'(\tilde g_i)R^i_{g_i}(-1)}&\not =0,
\end{align*}
where the numbers $\xi_{x,y}^i$ are defined by (\ref{defxime1}) and $\tilde g_i=\begin{cases} -h_i+c-1,& i=1,\ldots , m_1, \\k_{i-m_1}+1,& i=m_1+1,\ldots m. \end{cases}$
\end{proof}

\begin{corollary}\label{jodme}
Let $F_1$ and $F_2$ be two finite sets of positive integers (the empty set is allowed, in which case we take $\max F=-1$). For $a\not =0,1$ and $c\not =0,-1,-2,\ldots $, consider the weight $\rho _{a,c}^{F_1,F_2}$ defined by
$$
\rho _{a,c}^{F_1,F_2}=\prod_{f\in F_1}(x+c+f)\prod_{f\in F_2}(x-f)\rho_{a,c},
$$
where $\rho_{a,c}$ is the Meixner weight. Define $\tilde F_1=\{f_{1,M}-f+1,f\in F_1\}$, $h=\min F_1$ and $\tilde c=c+f_{1,M}+f_{2,M}+2$ (where as before $f_{i,M}=\max F_i$) and assume that
$$
\Omega_{a,\tilde c}^{J_{h}(\tilde F_1),I(F_2)}(n)\not =0, \quad n\ge 0.
$$
Then the measure $\rho _{a,c}^{F_1,F_2}$ has associated a sequence of orthogonal polynomials and Conjecture B is true.
\end{corollary}

\begin{proof}
It is easy to see that
$$
\rho _{a,c}^{F_1,F_2}=(1-a)^{c-\tilde c}\tilde \rho_{a,\tilde c}^{\tilde F_1,F_2,h}(x-f_{2,M}-1).
$$
The corollary is then a straightforward consequence of the previous Theorem.
\end{proof}

\section{Krawtchouk polynomials}\label{skr}
For $a\not=0,-1$, we write $(k_{n}^{a,N})_n$ for the sequence of  Krawtchouk polynomials defined by
\begin{equation}\label{defkrp}
k_{n}^{a,N}(x)=\frac{1}{n!}\sum_{j=0}^n(-1)^{n+j}\frac{(1+a)^{j-n}}{a^{j-n}}\frac{(-n)_j(-x)_j(N-n)_{n-j}}{j!}.
\end{equation}
Krawtchouk polynomials are eigenfunctions of the following second order difference operator ($n\ge 0$)
\begin{equation*}\label{sodekr}
D_{a,N} =\frac{-x\Sh_{-1}+(x-a(x-N+1))\Sh_0+a(x-N+1)\Sh_1}{1+a},\qquad D_{a,N} (k_{n}^{a,N})=nk_{n}^{a,N}.
\end{equation*}
For $a\not=0,-1$ and $N\not=1, 2, \ldots$, they are always orthogonal with respect to a moment functional $\rho _{a,N}$,
which we normalize by taking
$\langle \rho_{a,N},1\rangle=1$.
When $N$ is a positive integer and $a>0$, the first $N$ polynomials are orthogonal with respect to the positive Krawtchouk measure
\begin{equation*}\label{krw}
\rho_{a,N}=\frac{\Gamma(N)}{(1+a)^{N-1}}\sum _{x=0}^{N-1} \frac{a^{x}}{\Gamma(N-x)x!}\delta _x.
\end{equation*}
The structural formulas for $(k_{n}^{a,N})_n$ can be found in \cite{NSU}, pp. 30-53 (see
also \cite{KLS}, pp, 204-211).

The Krawtchouk case can be formally derived from the Meixner case taking into account that $k_n^{a,N}(x)=m_n^{-a,-N+1}(x)$. Anyway, since the corresponding proofs can be produced in a similar way to what we have already done, we omit the proofs and only display the results in this Section.

Consider now the algebra of operators $\A$ defined by \eqref{algdiffd}. Lemma 6.1 of \cite{du1} provides two $\D$-operators for Krawtchouk polynomials:
For $a\neq0,-1$, the operators $\mathcal{D}_1$ and $\mathcal{D}_2$ defined by \eqref{defTo} from the sequences $\varepsilon_n=1/(1+a)$ and $\tilde \varepsilon_n=-a/(1+a), n\geq0$, respectively, are $\mathcal{D}$-operators for the  Krawtchouk polynomials $(k_n^{a,N})_n$ \eqref{defkrp} and the algebra $\A$. More precisely
$$
\mathcal{D}_1=\frac{1}{1+a}\nabla,\quad \mathcal{D}_2=\frac{-a}{1+a}\Delta.
$$

\begin{theorem}\label{thkr} Let $F_1$ and $F_2$ be two finite sets of positive integers (the empty set is allowed, in which case we take $\max F=-1$). For $h\ge 1$, consider the transformed sets $I(F_1)=K=\{ k_{1},\ldots, k_{m_2}\}$ and $J_h(F_2)=H=\{ h_{1},\ldots, h_{m_1}\}$, where the involution $I$ and the transform $J_h$ are defined by (\ref{dinv}) and (\ref{dinv2}), respectively. Define $m=m_1+m_2$. Let $a$ and $N$ be real numbers satisfying $a\not =0,-1$; if $N$ is a positive integer, we also assume that $f_{1,M},f_{2,M}<N/2$ (so that $F_1\cap \{ N-1-f,f\in F_2\}=\emptyset$),
where $f_{i,M}=\max F_i$. Consider the Krawtchouk polynomials $(k_n^{a,N})_n$ (\ref{defkrp}). Assume that $\Omega_{a,N}^{K,H} (n)\not =0$ for $n\ge 0$ where the $m\times m$ Casorati determinant $\Omega _{a,N}^{K,H}$ is defined by
$$
\Omega_{a,N}^{K,H} (n)=\begin{vmatrix}
k^{a,-N}_{k_{1}}(-n) & \cdots &
k^{a,-N}_{k_{1}}(-n+m-1) \\
              \vdots & \ddots & \vdots \\
               \displaystyle
               k^{a,-N}_{k_{m_1}}(-n) & \cdots &k^{a,-N}_{k_{m_1}}(-n+m-1)
               \\ \vspace{-.3cm} \\
(-a)^{m-1}k^{1/a,-N}_{h_{1}}(-n) & \cdots &
k^{1/a,-N}_{h_{1}}(-n+m-1) \\
               \vdots & \ddots & \vdots \\
(-a)^{m-1}k^{1/a,-N}_{h_{m_2}}(-n)
               & \cdots &k^{1/a,-N}_{h_{m_2}}(-n+m-1)
             \end{vmatrix}.
$$
We then define the sequence of polynomials $(q_n)_n$ by
$$
q_n(x)=\begin{vmatrix}
(1+a)^m k^{a,N}_n(x) & -(1+a)^{m-1}k^{a,N}_{n-1}(x) & \cdots & (-1)^mk^{a,N}_{n-m}(x) \\
k^{a,-N}_{k_{1}}(-n-1) & k^{a,-N}_{k_{1}}(-n) & \cdots &
k^{a,-N}_{k_{1}}(-n+m-1) \\
               \vdots & \vdots & \ddots & \vdots \\
               k^{a,-N}_{k_{m_1}}(-n-1) & k^{a,-N}_{k_{m_1}}(-n) & \cdots &k^{a,-N}_{k_{m_1}}(-n+m-1)
               \\
(-a)^m k^{1/a,-N}_{h_{1}}(-n-1) & (-a)^{m-1}k^{1/a,-N}_{h_{1}}(-n) & \cdots &
k^{1/a,-N}_{h_{1}}(-n+m-1) \\
               \vdots & \vdots & \ddots & \vdots \\
(-a)^m k^{1/a,-N}_{h_{m_2}}(-n-1)  & (-a)^{m-1}k^{1/a,-N}_{h_{m_2}}(-n)
               & \cdots &k^{1/a,-N}_{h_{m_2}}(-n+m-1)
             \end{vmatrix}.
$$
Then

\noindent
(1) If $N$ is not a positive integer, then the polynomials $(q_n)_n$ are orthogonal with respect to the measure
$$
\tilde \rho^{F_1,F_2,h}_{a,N}=\prod_{f\in F_1}(x+f_{1,M}+1-f)\prod_{f\in F_2}(N-x-1+f)\rho_{a,N+f_{1,M}+f_{2,M}+h+1}(x+f_{1,M}+1).
$$
If $N$ is a positive integer, then the polynomials $(q_n)_{0\le n\le N-1}$ are orthogonal with respect to the same measure.

\noindent
(2) The polynomials $(q_n)_n$ are eigenfunctions of a higher order difference operator of the form (\ref{doho}) with
$$
-s=r=\sum _{f\in F_1}f-\sum _{f\in F_2}f-\frac{n_{F_1}(n_{F_1}-1)}{2}-\frac{n_{F_2}(n_{F_2}-1)}{2}+n_{F_2}(f_{2,M}+h)+1
$$
(which can be explicitly constructed using Theorem \ref{Teor1}).
\end{theorem}

\begin{corollary}\label{jodkr}
Let $F_1$ and $F_2$ be two finite sets of positive integers (the empty set is allowed). We assume that $a\not =0,-1$,
and if $N$ is a positive integer, we assume in addition that $f_{1,M},f_{2,M}<N/2$ (where $f_{i,M}=\max F_i$).
Consider the weight $\rho _{a,N}^{F_1,F_2}$ defined by
\begin{equation}\label{defmkre}
\rho _{a,N}^{F_1,F_2}=\prod_{f\in F_1}(x-f)\prod_{f\in F_2}(N-1-f-x)\rho_{a,N},
\end{equation}
where $\rho_{a,N}$ is the Krawtchouk weight. Define $\tilde F_2=\{f_{2,M}+1-f,f\in F_2\}$, $h=\min F_2$ and $\tilde N=N-f_{1,M}-f_{2,M}-2$ and assume that
$$
\Omega_{a,\tilde N}^{I(F_1),J_{h}(\tilde F_2)}(n)\not =0, \quad n\ge 0.
$$
Then the measure $\rho _{a,N}^{F_1,F_2}$ has associated a sequence of orthogonal polynomials and Conjecture C is true.
\end{corollary}

Notice that when $N$ is a positive integer, there are different sets $F_1$ and $F_2$ for which the measures
(\ref{defmkre}) are equal, but only one of these possibilities satisfies the assumption $f_{1,M},f_{2,M}<N/2$. Actually,
we only need the weaker assumption $F_1\cap \{ N-1-f,f\in F_2\}=\emptyset$ to prove both Theorem \ref{thkr} and Corollary \ref{jodkr}. However, the stronger assumption $f_{1,M},f_{2,M}<N/2$ is better because, in addition, it minimizes the order $2r$ of the difference operator with respect to which the polynomials $(q_n)_n$ are eigenfunctions.
This fact will be clear with an example. Take $N=100$ and the measure $\mu=(x-1)(x-5)(x-68)\rho_{a,N}$. There are eight couples of different sets $F_1$ and $F_2$ for which the measures $\mu$ and $\tilde \rho _{a,N}^{F_1,F_2}$ (\ref{defmkre}) coincide (except for a sign). They are the following
\begin{align*}
F_1&=\{1,5,68\}, F_2=\emptyset, \quad &F_1&=\emptyset, F_2=\{31,94,98\},\\
F_1&=\{1,5\}, F_2=\{31\},\quad &F_1&=\{1,68\}, F_2=\{94\},\quad &F_1&=\{5,68\}, F_2=\{98\},\\
F_1&=\{1\}, F_2=\{31,94\},\quad &F_1&=\{5\}, F_2=\{31,98\},\quad &F_1&=\{68\}, F_2=\{94,98\}.
\end{align*}
Only one of these couples satisfies the assumption $f_{1,M},f_{2,M}<N/2$: $F_1=\{1,5\}$, $F_2=\{31\}$. Actually, it is easy to check that this couple
minimizes the number
$$
r=\sum_{f\in F_1}f+\sum_{f\in F_2}f-\frac{n_{F_1}(n_{F_1}-1)}{2}-\frac{n_{F_2}(n_{F_2}-1)}{2}+1.
$$
Hence, it also minimizes de order $2r$ of the difference operator with respect to which the polynomials $(q_n)_n$ are eigenfunctions.

\section{Proofs of the results in Section 4}\label{sproofs}
In this Section, we include the proofs of Theorem \ref{Teor1}, Lemma \ref{lgp1} and Lemma \ref{lgp2} in Section 4.

\begin{proof} of Theorem \ref{Teor1}.

First, for  $j=1,2,\ldots,m,$ and $i\in\mathbb{Z}$, we introduce the determinants $\Omega_{i,j}(n)$, $n\ge 0$, by the $m\times m$ determinant defined in the same way as $\Omega (n-i)$ but replacing the $j$-th column of $\Omega (n-i)$ by the column vector
$$
\left(\xi_{n-1,m+i-1}^1R_1(n-1),\ldots ,\xi_{n-1,m+i-1}^mR_m(n-1)\right)^t.
$$
Observe that $\Omega_{0,1}(n)=\Omega(n)$.

\medskip
We now write the polynomials $(q_n)_n$ in a more convenient way. Indeed
let $\phi_{n,j}, j=1,\ldots,m,$ be the (unique) solutions of the following system of equations
\begin{equation}\label{SyEq}
  \sum_{j=1}^m\xi_{n-j,m-j}^iR_i(n-j)\phi_{n,j}=\xi_{n,m}^iR_i(n),\quad i=1,2,\ldots,m.
\end{equation}
They can be written as
$$
  \phi_{n,j}=\frac{\Omega_{1,j}(n+1)}{\Omega (n)},\quad j=1,2,\ldots,m.
$$
For $j=1,2,\ldots,m$, define the sequences of numbers $\beta_{n,j}$  by
\begin{equation}\label{betas}
\beta_{n,j}=\begin{cases}(-1)^{j+1}\phi_{n,j}, & n\geq j.\\
0, & n< j,\end{cases}
\end{equation}
It is then clear that
\begin{equation}\label{qus2}
q_n=\Omega(n)\left(p_n+\sum_{j=1}^m\beta_{n,j}p_{n-j}\right).
\end{equation}

\bigskip

The definition of \eqref{Dq} gives
\begin{align*}
D_{q,S}(p_n)=&\lambda_np_n+\sum_{h=1}^m\sum_{i=1}^nM_h(n-i)(-1)^{i+1}\varepsilon_{n}^h\cdots\varepsilon_{n-i+1}^hR_h(n)p_{n-i}\\
=&\lambda_np_n+\lambda_{n,1}p_{n-1}-\lambda_{n,2}p_{n-2}+\cdots+(-1)^{m+1}\lambda_{n,m}p_{n-m}+\sum_{i=m+1}^n(-1)^{i+1}\lambda_{n,i}p_{n-i}
\end{align*}
where
\begin{equation}\label{lambj}
 \lambda_{n}=P_S(n),\quad  \lambda_{n,i}=\sum_{h=1}^m\varepsilon_{n}^h\cdots\varepsilon_{n-i+1}^hM_h(n-i)R_h(n),\quad i=1,2,\ldots,n.
\end{equation}

Taking into account the identity (\ref{qus2}) for the polynomials $(q_n)_n$ we get
\begin{align*}
D_{q,S}\left(\frac{q_n}{\Omega(n)}\right) =&D_{q,S}(p_n)+\beta_{n,1}D_{q,S}(p_{n-1})+\cdots+\beta_{n,m}D_{q,S}(p_{n-m})\\
=&\lambda_np_n+(\lambda_{n,1}+\beta_{n,1}\lambda_{n-1})p_{n-1}+(-\lambda_{n,2}+\beta_{n,1}\lambda_{n-1,1}+\beta_{n,2}\lambda_{n-2})p_{n-2}\\
&+\left(\lambda_{n,3}-\beta_{n,1}\lambda_{n-1,2}+\beta_{n,2}\lambda_{n-2,1}+\beta_{n,3}\lambda_{n-3}\right)p_{n-3}+\cdots\\
&+\left((-1)^{m+1}\lambda_{n,m}+(-1)^m\beta_{n,1}\lambda_{n-1,m-1}+\cdots+\beta_{n,m}\lambda_{n-m}\right)p_{n-m}\\
&+\sum_{i=m+1}^n\left((-1)^{i+1}\lambda_{n,i}+(-1)^{i}\beta_{n,1}\lambda_{n-1,i-1}+\cdots+\beta_{n,m}\lambda_{n-m,i-m}\right)p_{n-i}.
\end{align*}
We must have the identities
\begin{align*}
(-1)^{i+1}\lambda_{n,i}+(-1)^{i}\beta_{n,1}\lambda_{n-1,i-1}+\cdots+\beta_{n,m}\lambda_{n-m,i-m}=&\lambda_n\beta_{n,i},\quad i=1,2,\ldots,m,\\
(-1)^{i+1}\lambda_{n,i}+(-1)^{i}\beta_{n,1}\lambda_{n-1,i-1}+\cdots+\beta_{n,m}\lambda_{n-m,i-m}=&0,\qquad i=m+1,\ldots,n.&
\end{align*}
Using the definition of $\beta_{n,i}$ in \eqref{betas} the previous identities are equivalent to
\begin{align}\label{ident}
\nonumber\lambda_{n,i}&=\lambda_{n-1,i-1}\phi_{n,1}+\lambda_{n-2,i-2}\phi_{n,2}+\cdots+\lambda_{n-i+1,1}\phi_{n,i-1}+(\lambda_n-\lambda_{n-i})\phi_{n,i},\quad i=1,2,\ldots,m,\\
\lambda_{n,i}&=\lambda_{n-1,i-1}\phi_{n,1}+\lambda_{n-2,i-2}\phi_{n,2}+\cdots+\lambda_{n-m,i-m}\phi_{n,m},\quad i=m+1,\ldots,n.
\end{align}

Before checking the identities \eqref{ident}, we will assume the following claims, which will be proved at the end.

\textsc{Claim I}: The coefficients $\lambda_{n,i}$ defined by \eqref{lambj} can be written as
\begin{equation}\label{claim1}
   \lambda_{n,i}=\sum_{h=1}^{m}S(n-i+h)\Omega_{i-h+1,h}(n+1),\quad i=1,2,\ldots,n.
\end{equation}

\textsc{Claim II}: We have that
\begin{equation}\label{claim2}
   \lambda_n-\lambda_{n-i}=\sum_{h=1}^iS(n-h+1)\Omega(n-h+1)=\sum_{h=1}^iS(n-i+h)\Omega_{-h+1,h}(n-i+1),\quad i=1,2,\ldots,m.
\end{equation}

\textsc{Claim III}: We have the following $m$ formulas
\begin{equation}\label{claim3}
 \sum_{h=1}^{m}\Omega_{l-h,j}(n-h+1)\phi_{n,h}=\Omega_{l,j}(n+1),\quad j=1,2,\ldots,m,\quad l=1,2,\ldots,n.
\end{equation}

\textsc{Claim IV}: We have that
\begin{equation}\label{claim4}
 \Omega_{1-h,j}(n)=0,\quad h\in\{1,2,\ldots,m\}\setminus\{j\}.
\end{equation}

\bigskip

The first identity in \eqref{ident} follows using \eqref{claim1} and \eqref{claim2} since, for $i=1,2,\ldots,m,$ we have that
\begin{align*}
&\lambda_{n-1,i-1}\phi_{n,1}+\lambda_{n-2,i-2}\phi_{n,2}+\cdots+\lambda_{n-i+1,1}\phi_{n,i-1}+(\lambda_n-\lambda_{n-i})\phi_{n,i}\\
&=\left(\sum_{h=1}^{m}S(n-i+h)\Omega_{i-h,h}(n)\right)\phi_{n,1}+\left(\sum_{h=1}^{m}S(n-i+h)\Omega_{i-h-1,h}(n-1)\right)\phi_{n,2}+\cdots+\\
&\qquad\quad +\left(\sum_{h=1}^mS(n-i+h)\Omega_{2-h,h}(n-i+2)\right)\phi_{n,i-1} +\left(\sum_{h=1}^iS(n-i-h)\Omega_{-h+1,h}(n-i+1)\right)\phi_{n,i}.
\end{align*}
The previous identity can be rewritten, using \eqref{claim4}, as
\begin{align*}
&S(n-i+1)\sum_{h=1}^m\Omega_{i-h,1}(n-h+1)\phi_{n,h}+S(n-i+2)\sum_{h=1}^m\Omega_{i-h-1,2}(n-h+1)\phi_{n,h}+\cdots+\\
&\qquad\quad +S(n)\sum_{h=1}^m\Omega_{-h+1,i}(n-h+1)\phi_{n,h}.
\end{align*}
Finally, we use \eqref{claim3} for $j=1,l=i$ in the first term of the previous expression, \eqref{claim3} for $j=2,l=i-1$ in the second, and so on, until the last one, which we use \eqref{claim3} for $j=i,l=1$. Therefore we obtain, using again \eqref{claim1},
\begin{equation*}
S(n-i+1)\Omega_{i,1}(n+1)+S(n-i+2)\Omega_{i-1,2}(n+1)+\cdots+S(n)\Omega_{1,i}(n+1)=\lambda_{n,i}.
\end{equation*}

The second identity in \eqref{ident} is very similar but now $i=m+1,\ldots,n$. Using \eqref{claim1} we get
\begin{align*}
&\lambda_{n-1,i-1}\phi_{n,1}+\lambda_{n-2,i-2}\phi_{n,2}+\cdots+\lambda_{n-m,i-m}\phi_{n,m}\\
&=\left(\sum_{h=1}^{m}S(n-i+h)\Omega_{i-h,h}(n)\right)\phi_{n,1}+\left(\sum_{h=1}^{m}S(n-i+h)\Omega_{i-h-1,h}(n-1)\right)\phi_{n,2}+\\
&\qquad\quad +\cdots+\left(\sum_{h=1}^{m}S(n-i+h)\Omega_{i-h-m+1,h}(n-m+1)\right)\phi_{n,m}.
\end{align*}
Again, rewriting everything as before and using \eqref{claim3} for $j=1,l=i$ in the first term, \eqref{claim3} for $j=2,l=i-1$ in the second, and so on, until the last one, which we use \eqref{claim3} for $j=m,l=i-m+1$, we get, using \eqref{claim1}, $\lambda_{n,i}$.

\medskip

Finally we will prove the claims. Claim IV follows from the definition of $\Omega_{i,j}(n)$ since for fixed $j$ and $n$, for $h\neq j$, we observe that there will always be two repeated columns, so the determinant is zero.

For the Claim III, let us write first $\Omega_{i,j}(n)$ as the expansion of the determinant with respect to the $j$-th column
\begin{equation}\label{ExpDnij}
  \Omega_{i,j}(n)=\xi_{n-1,m+i-1}^1R_1(n-1)C_{1,j}(n-i)+\cdots+\xi_{n-1,m+i-1}^mR_m(n-1)C_{m,j}(n-i),
\end{equation}
where $C_{h,j}(n-i), h=1,2,\ldots,m,$ are the $(h,j)$-cofactors of the matrix given by  $\Omega_{i,j}(n)$. Observe that, by definition, for $h,j=1,2,\ldots,m,$
\begin{equation}\label{cofac}
C_{h,j}(n-i)=(-1)^{h+j}\det\left(\xi_{n-i-r,m-r}^lR_l(n-i-r)\right)_{\scriptsize \left\{\begin{array}{l}
                                                           l\in\{1,2,\ldots,m\}\setminus\{h\} \\
                                                            r\in\{1,2,\ldots,m\}\setminus\{j\}
                                                          \end{array}\right\}}.
\end{equation}
Expanding $\Omega_{l-h,j}(n-h+1)$ as in \eqref{ExpDnij} we obtain, for $h=1,2,\ldots,m$,
\begin{align*}
\Omega_{l-h,j}(n-h+1)&=\xi_{n-h,m+l-h-1}^1R_1(n-h)C_{1,j}(n-l+1)\\
&\qquad\qquad +\cdots+\xi_{n-h,m+l-h-1}^mR_m(n-h)C_{m,j}(n-l+1),
\end{align*}
where the cofactors  $C_{1,j}(n-l+1),\ldots,C_{m,j}(n-l+1)$ are independent of $h$. Multiplying the first equation in \eqref{SyEq} by $\xi_{n-m,l-1}^1C_{1,j}(n-l+1)$, the second equation in \eqref{SyEq} by $\xi_{n-m,l-1}^2C_{2,j}(n-l+1)$, and so on, and summing up these last $m$ equations we obtain \eqref{claim3}, using again \eqref{ExpDnij}.

Claim II is a consequence of \eqref{Pgs}, \eqref{lambj} and the fact that for a fixed index $r$ the sequences $\Omega_{1-h,h}(n-r-h)$ are all equal for $h=1,2,\ldots,m$, which follows straightforwardly from the definition of $\Omega_{i,j}(n)$.

Finally, let us proof Claim I. First, we observe from the definition of \eqref{emeiexp} that

\begin{align*}
M_h(n-i)=&\sum_{j=1}^m(-1)^{h+j}S(n-i+j)\xi_{n-i,m-j}^h\times\\
&\quad\times\det\left(\xi_{n-i-r,m-j-r}^lR_l(n-i-r)\right)_{\scriptsize \left\{\begin{array}{l}
                                                           l\in\{1,2,\ldots,m\}\setminus\{h\} \\
                                                            r\in\{-j+1,-j+2,\ldots,m-j\}\setminus\{0\}
                                                          \end{array}\right\}}\\
=&\sum_{j=1}^m(-1)^{h+j}S(n-i+j)\xi_{n-i,m-j}^h\times\\
&\quad\times\det\left(\xi_{n-i+j-r,m-r}^lR_l(n-i+j-r)\right)_{\scriptsize \left\{\begin{array}{l}
                                                           l\in\{1,2,\ldots,m\}\setminus\{h\} \\
                                                            r\in\{1,2,\ldots,m\}\setminus\{j\}
                                                          \end{array}\right\}}\\
=&\sum_{j=1}^mS(n-i+j)\xi_{n-i,m-j}^hC_{h,j}(n-i+j).
\end{align*}
The last step is a consequence of the definition of a cofactor matrix given by \eqref{cofac}.

Now, using the previous considerations in \eqref{lambj} we get
\begin{align*}
\lambda_{n,i}&=\sum_{h=1}^m\varepsilon_{n}^h\cdots\varepsilon_{n-i+1}^hM_h(n-i)R_h(n)=\sum_{h=1}^m\xi_{n,i}^hM_h(n-i)R_h(n)\\
&=\sum_{h=1}^m\xi_{n,i}^h\left(\sum_{j=1}^mS(n-i+j)\xi_{n-i,m-j}^hC_{h,j}(n-i+j)\right)R_h(n)\\
&=\sum_{j=1}^mS(n-i+j)\left(\sum_{h=1}^m\xi_{n,m-j+i}^hR_h(n)C_{h,j}(n-i+j)\right)\\
&=\sum_{j=1}^{m}S(n-i+j)\Omega_{i-j+1,j}(n+1).
\end{align*}
The simplification of the coefficients $\xi_{n,i}^h$ follows using the definition \eqref{defxi}. The last step is a consequence of \eqref{ExpDnij}. Therefore we get \eqref{claim1}.

\hfill \end{proof}

\bigskip

\begin{proof} of Lemma \ref{lgp1}.

To simplify the notation, we write $\mu_i =\deg R_i$ and  $\phi_{\mu_1,\ldots , \mu _m}$ for the special case of the determinant $\det  (S_{i,j})_{i,j=1}^m$ when $R_i(x)=x^{\mu _i}$. Without loosing generality, we can also assume that $l_i<l_j$, $1\le i<j\le m$.
We proceed in three steps.

\bigskip

\noindent\textbf{First step.} \textit{Assume that the Lemma holds for the case when $R_i(x)$ are powers of $x$. Then it also holds
for any polynomials $R_i$.}
For arbitrary polynomials $R_i$ with leading coefficient equal to $\alpha _i$, we can write
$$
\det (S_{i,j})_{i,j=1}^m=\left(\prod_{i=1}^m\alpha_i\right)\phi_{\mu_1,\ldots , \mu _m}+\sum a_{\nu_1,\ldots , \nu _m}\phi_{\nu_1,\ldots , \nu _m},
$$
where the sum is taken over all $(\nu_1,\ldots , \nu _m)$ satisfying that $0\le \nu_i\le \mu_i$, $i=1,\ldots ,m$, and
$\nu_{i_0}< \mu_{i_0}$ at least for some  $i_0$, $1\le i_0\le m$. The step now follows easily.

\bigskip

\noindent\textbf{Second step} \textit{The Lemma holds for the case when $R_i(x)$ are any powers of $x$ and $m_1=0$ or $m_2=0$.}
It is enough to prove it for $m_2=0$, that is, $m_1=m$. By interchanging rows in the determinant,
we can assume that the degrees $\mu_i$, $i=1,\ldots , m$, are decreasing. Since they are different, we get that $\mu_j\ge m-j$.

We use Schur functions $s_\lambda(x_1,\ldots , x_m)$. Following the notation of \cite{Macd}, we write
$$
s_\lambda (x_1,\ldots ,x_m)=\frac{\det (x_i^{\lambda _j+m-j})_{1\le i,j\le m}}{\prod_{1\le i<j\le m}(x_i-x_j)},
$$
where $\lambda =(\lambda _1,\ldots , \lambda _m)$ is a partition of nonnegative integers in decreasing order: $\lambda_j\ge \lambda _{j+1}$. As usual we denote the length of the partition $\lambda$ by $\vert \lambda \vert =\lambda_1+\cdots +\lambda_m$.
Hence, by writing $\lambda_j=\mu_j-m+j$ and $x_i=x-l_i$, we get
\begin{equation}\label{md1}
\det (R_{j}(x-l_i))_{i,j=1}^m=s_\lambda (x_1,\ldots , x_m)\prod_{1\le i<j\le m}(x_i-x_j).
\end{equation}
We have to consider the complete symmetric functions $h_r$, $r\ge 0$, defined by
$$
h_r(x_1,\ldots , x_m)=\sum _{\vert\tau \vert=r}m_\tau (x_1,\ldots , x_m),
$$
where $m_\tau =\sum _{\alpha=(\alpha_1,\ldots ,\alpha_m)} x_1^{\alpha_1}\cdots x_m^{\alpha_m}$, summed over all distinct permutations $\alpha$ of
$\tau = (\tau_1,\ldots , \tau_m)$. Using (3.4), p. 41 of \cite{Macd}, we have
\begin{equation}\label{md2}
s_\lambda (x_1,\ldots , x_m)=\det(h_{\mu_i +j-m}(x_1,\ldots , x_m))_{1\le i,j\le m}.
\end{equation}
Since $x_i=x-l_i$, it is easy to see that $h_{\mu_i+j-m}(x_1,\ldots ,x_m)$ is a polynomial in $x$ of degree $\mu_i+j-m$ with leading coefficient equal to $h_{\mu_i+j-m}(1,\ldots ,1)=\binom{\mu_i+j-1}{\mu_i+j-m}$ (\cite{Macd}, p. 26).
(\ref{md1}) and (\ref{md2}) then give that  $\det (R_{j}(x-l_i))_{i,j=1}^m$ is a polynomial of degree at most
$H=\sum_{i=1}^m\mu_i -\frac{m(m-1)}{2}$ with  coefficient of the power $x^H$ equal to
$$
\det \left( \binom{\mu_i+j-1}{\mu_i+j-m}\right) _{1\le i,j\le m}\prod_{1\le i<j\le m}(l_j-l_i).
$$
This determinant can be easily transformed into a Vandermonde determinant
by elementary column operations to get
$$
(-1)^{\binom{m}{2}}\frac{\left(\prod_{1\le i<j\le m}(l_j-l_i)\right)\left(\prod_{1\le i<j\le m}(\mu _j-\mu _i)\right)}{\prod _{i=1}^{m-1}i!},
$$
which it is clearly different to $0$.

\bigskip

\noindent\textbf{Third step} \textit{The Lemma holds for the case when $R_i(x)$ are any powers of $x$.}
Write $H=\sum_{i=1}^m\mu_i -m_1(m_1-1)/2-m_2(m_2-1)/2$. As a function of $w$, the determinant $C_x(w )=\det (S_{i,j}(x))_{i,j=1}^m$ is a polynomial  of
degree at most $K=m_2(m-(m_2+1)/2)$:
$$
C_x(w)=\sum_{j=0}^{K}a_j(x)w ^j.
$$
Denoting by $f_{x,l}(w)$ the columns of the determinant $C_x(w)=\det (S_{i,j}(x))_{i,j=1}^m$, we get
$$
a_j(x)=\frac{C_x^{(j)}(0)}{j!}=\sum _{h_1+\cdots +h_m=j}\frac{1}{h_1!\cdots h_m!}\det \begin{pmatrix}f_{x,1}^{(h_1)}(0)& \hdots &f_{x,m}^{(h_m)}(0)\end{pmatrix},
$$
where all the derivatives are with respect to $w$.

Each number $j$, $m_2(m_2-1)/2\le j\le K$, can be written  as $j=j_1+\cdots + j_{m_2}$, where $0\le j_1<j_2<\cdots<j_{m_2}\le m-1$. Write $i_u$, $u=1,\ldots , m_1$, for the numbers $\{0,1,2,\ldots , m-1\}\setminus \{j_1,\ldots, j_{m_2}\}$ ordered in increasing size. Denote by $D_{j_1,\ldots ,j_{m_2}}(x)$ the determinant with $l$-th column equal to $f_{x,l}^{(m-l)}(0)$, if $l=m-j_u$, $u=1,\ldots , m_2$, and $f_{x,l}(0)$ if $l=m-i_u$, $u=1,\ldots , m_1$.
A careful computation shows then that $a_j(x)=0$, for $j<m_2(m_2-1)/2$ and that
$$
a_{j}(x)=\sum _{\substack{j=j_1+\cdots + j_{m_2}\\ 0\le j_1<j_2<\cdots<j_{m_2}\le m-1}}D_{j_1,\ldots ,j_{m_2}}(x), \quad m_2(m_2-1)/2\le j\le K.
$$
By interchanging columns and deleting the corresponding entries of its last column ($(l,m)$ for $l=1,\ldots , m_1$ or $(l,m)$ for $l=m_1+1,\ldots , m$, depending whether $j_1=0$ or not), we see that each one of these determinants can be written as a $2\times 2$ block determinant
$$
(-1)^{j-\binom{m_2}{2}}\begin{pmatrix}A&0\\0&B\end{pmatrix},
$$
where
$$
A=\det (R_u(x-l_{k}))_{\substack{1\le u\le m_1\\k=m-i_{m_1},\ldots ,m-i_1}}\quad B=\det (R_u(x-l_{k}))_{\substack{m_1+1\le u\le m\\k=m-j_{m_2},\ldots,m-j_1}}.
$$
Applying the second step, we get that each $a_j(x)$ is a polynomial in $x$ of degree $H$ with leading coefficient equal to
$$
\alpha_j=(-1)^jc\sum _{\substack{j=j_1+\cdots + j_{m_2}\\ 0\le j_1<j_2<\cdots<j_{m_2}\le m-1}}\prod_{1\le u<k\le m_1}(l_{m-i_u}-l_{m-i_k})\prod_{1\le u<k\le m_2}(l_{m-j_u}-l_{m-j_k}),
$$
where $c$ is the following number (independent of $j$)
$$
(-1)^{\binom{m_1}{2}}\frac{\prod_{1\le u<k\le m_1}(\mu _k-\mu _u)}{\prod _{u=1}^{m_1-1}u!}\frac{\prod_{m_1+1\le u<k\le m}(\mu _k-\mu _u)}{\prod _{u=1}^{m_2-1}u!}.
$$
Hence $a_j(x)\not =0$, for $m_2(m_2-1)/2 \le j\le K$, and then as a function of $x$, the determinant $\det (S_{i,j})_{i,j=1}^m$ is a polynomial in $x$ of degree at most $H$ with  coefficient of $H$ equal to
$$
\sum_{j=m_2(m_2-1)/2}^{K}\alpha_j w^j.
$$
Since this is a polynomial in $w$ of degree $K$, it vanishes for at most a finite set of complex numbers $w$.
In particular, for $l_i=i$, a careful computation gives that this polynomial is equal to
$$
(-1)^{\binom{m_1}{2}}(-1)^{\binom{m_2}{2}}\left(\prod_{1\le u<k\le m_1}(\mu _k-\mu _u)\right)\left(\prod_{m_1+1\le u<k\le m}(\mu _k-\mu _u)\right) w ^{\binom{m_2}{2}}(w -1)^{m_1m_2}.
$$
This proves the third step and the Lemma.

\end{proof}

\begin{proof} of Lemma \ref{lgp2}.

We first notice that the cases $\epsilon =0$ and $\epsilon =1$ are equivalent. Indeed, write
$\mathcal{R}$, $\tilde{\mathcal{R}}$ for the $n$-tuple of polynomials
$$
\mathcal{R}=(R_1,\cdots , R_{m_1},R_{m_1+1},\cdots , R_{m}),\quad \tilde{\mathcal{R}}=(R_{m_1+1},\cdots , R_{m},R_1,\cdots , R_{m_1}),
$$
respectively. Write also
$$
U^{\mathcal{R},m_1,m_2}_{\epsilon, w}=\sum _{j=1}^{m+1} (-1)^{j+1}w^{\epsilon (m+1-j)}\det U_j(x).
$$
Then it is easy to see that $U^{\mathcal{R},m_1,m_2}_{1, w}=(-1)^{m_1}w^{\binom{m+1}{2}}U^{\tilde{\mathcal{R}},m_2,m_1}_{0,1/w}$. Hence, we only prove the case $\epsilon =0$.

Proceeding as in the previous proof, it is enough to prove the case when $m_2=0$ and $R_i$ are any powers of $x$.
We denote again $\mu_i=\deg R_i$.

Write $f_{i,j}$, $i=1,\ldots , m$, for the rows of the matrix $U_j$ and $\varphi (x)=\sum _{j=1}^{m+1} (-1)^{j+1}\det U_j(x)$. We then have
$$
\varphi^{(h)}(x)=\sum _{h_1+\cdots +h_m=h}\frac{h!}{h_1!\cdots h_m!}\sum _{j=1}^{m+1}(-1)^{j+1}\det \begin{pmatrix}f_{1,j}^{(h_1)}(x)\\ \vdots \\f_{m,j}^{(h_m)}(x)\end{pmatrix}.
$$
Then, if $\varphi ^{(h)}\not =0$ there exist nonnegative integers $h_1,\ldots, h_m$, $h_1+\cdots +h_m=h$ for which
\begin{equation}\label{md3}
\sum _{j=1}^{m+1}(-1)^{j+1}\det \begin{pmatrix}f_{1,j}^{(h_1)}(x)\\ \vdots \\f_{m,j}^{(h_m)}(x)\end{pmatrix}\not=0.
\end{equation}
Since $f_{i,j}(x)=((x-r)^{\mu_i})_{r=1-j,2-j,\ldots,m+1-j, r\not =0}$, this implies that $\mu_i-h_i\ge 0$, $i=1,\ldots, m$, and different to each others. We can assume that they are decreasing numbers.  In particular, $\mu_i-h_i\ge m-i$, $i=1,\ldots, m$.

We use again Schur functions. Write $l_i^j$, $i=1,\ldots , m$, for the numbers in the set $\{1-j,2-j,\ldots,m+1-j\}\setminus \{0\}$ in increasing order.
Define now $\lambda_i=\mu_i-h_i-m+i$ and $x_i^j=x-l_i^j$, $i=1,\ldots , m$. This gives
\begin{equation}\label{md4}
s_\lambda (x_1^j,\ldots , x_m^j)\prod_{1\le i<k\le m}(x_i^j-x_k^j)=\left(\prod_{i=1}^mh_i!\binom{\mu_i}{h_i}\right) \det \begin{pmatrix}f_{1,j}^{(h_1)}(x)\\ \vdots \\f_{m,j}^{(h_m)}(x)\end{pmatrix}.
\end{equation}
We notice that, as functions of $j$, each symmetric function
$$
m_\tau (x_1^j,\ldots , x_m^j)=\sum _{\alpha=(\alpha_1,\ldots ,\alpha_m)} (x_1^j)^{\alpha_1}\cdots (x_m^j)^{\alpha_m},
$$
is a polynomial in $j$ of degree $\vert \tau\vert$. Hence each complete symmetric function $h_r(x_1^j,\ldots , x_m^j)$ is also a polynomial in $j$ of degree $r$. To stress this dependence on $j$ we write
$h_r(x_1^j,\ldots , x_m^j)=\phi_r(j)$. As usual, we take $\phi_r=0$ if $r<0$.
Using (3.4), p. 41 of \cite{Macd}, we write
\begin{align*}
s_\lambda (x_1^j,\ldots , x_m^j)&=\det(h_{\mu_i -h_i+k-m}(x_1^j,\ldots , x_m^j))_{1\le i,k\le m}\\&=\det(\phi_{\mu_i -h_i+k-m}(j))_{1\le i,k\le m}.
\end{align*}
(\ref{md3}) and (\ref{md4}) then give
$$
\sum _{j=1}^{m+1}(-1)^{j+1}\det(\phi_{\mu_i -h_i+k-m}(j))_{1\le i,k\le m}\prod_{1\le i<k\le m}(x_i^j-x_k^j)\not=0.
$$
A simple computation gives that
$$
\prod_{1\le i<k\le m}(x_i^j-x_k^j)=\binom{m}{j-1}\prod_{i=1}^{m-1}i! ,
$$
and then
$$
\sum _{j=1}^{m+1}(-1)^{j+1}\binom{m}{j-1}\det(\phi_{\mu_i -h_i+k-m}(j))_{1\le i,k\le m}\not=0.
$$
Since $\phi_{\mu_i -h_i+k-m}(j)$ is a polynomial of degree $\mu_i-h_i+k-m$, we deduce that
$\det(\phi_{\mu_i -h_i+k-m}(j))_{1\le i,k\le m}$ is a polynomial in $j$ of degree at most $\sum_i(\mu_i-h_i)-m(m-1)/2$.
Taking into account that $\sum_{j=1}^{m+1}(-1)^{j+1}\binom{m}{j-1}p(j)=0$, for all polynomial $p$ with $\deg (p)\le m-1$,
we then conclude that $\sum_i(\mu_i-h_i)-m(m-1)/2\ge m$. Since $\sum_ih_i=h$, this gives
$\sum_i\mu_i-m(m+1)/2\ge h$. In other words, we have proved that if $\varphi ^{(h)}(x)\not =0$ then $\sum_i\mu_i-m(m+1)/2\ge h$. This proves that
the degree of $\varphi (x)$ is at most $\sum_i\mu_i-m(m+1)/2$.

\end{proof}

\end{document}